\documentclass[a4paper,11pt]{amsart}

\usepackage{latexsym}
\usepackage{amsmath}
\usepackage{amsthm}
\usepackage{amssymb}
\usepackage{amsfonts}
\usepackage{amsxtra}
\usepackage{mathrsfs}
\usepackage{setspace}
\usepackage[all]{xypic}
\usepackage{enumerate}

\newcommand{\ff}{\footnote}

\newcommand{\nc}{\newcommand}

\newtheorem{thm}{Theorem}
\newtheorem{prop}[thm]{Proposition}
\newtheorem{cor}[thm]{Corollary}

\newtheorem{lem}[thm]{Lemma}

\theoremstyle{definition}
\newtheorem{defn}[thm]{Definition}

\newtheorem{remark}[thm]{Remark}

\numberwithin{thm}{section}

\nc{\bdm}{\begin{displaymath}}
\nc{\edm}{\end{displaymath}}
\nc{\bthm}{\begin{thm}}
\nc{\ethm}{\end{thm}}
\nc{\blem}{\begin{lem}}
\nc{\elem}{\end{lem}}
\nc{\bcor}{\begin{cor}}
\nc{\ecor}{\end{cor}}
\nc{\beq}{\begin{equation}}
\nc{\eeq}{\end{equation}}
\nc{\bprop}{\begin{prop}}
\nc{\eprop}{\end{prop}}
\nc{\bdefn}{\begin{defn}}
\nc{\edefn}{\end{defn}}

\newcommand{\Z}{\mathbb{Z}}
\newcommand{\N}{\mathbb{N}}
\newcommand{\Q}{\mathbb{Q}}
\newcommand{\R}{\mathbb{R}}
\newcommand{\C}{\mathbb{C}}

\nc{\T}{\mathbb{T}}
\renewcommand{\O}{\mathcal{O}}
\nc{\Proj}{\mathbb{P}^1}

\nc{\Y}{\mathbb{Y}}
\nc{\X}{\mathbb{X}}
\nc{\mf}{\mathfrak}
\nc{\mc}{\mathcal}
\nc{\ms}{\mathscr}
\nc{\W}{\mathscr{W}}
\nc{\D}{\mathscr{D}}
\nc{\AW}{\ms{A}}
\nc{\ds}{\dots}
\renewcommand{\Rp}{\textrm{Rep}(Q,\nu)}
\nc{\Hom}{\mathrm{Hom}}
\nc{\p}{\partial}
\nc{\Rpbar}{\textrm{Rep}(\bar{Q},\nu)}
\nc{\h}{\mf{h}}
\nc{\hr}{\mf{h}_{\textrm{reg}}}
\renewcommand{\deg}{\textrm{deg}}
\nc{\B}{\ms{B}}
\renewcommand{\L}{\ms{L}}

\nc{\Rad}{\mf{R}}
\nc{\mbf}{\mathbf}
\nc{\biP}[2]{ {}_{#1} \ms{P}_{#2}}
\nc{\Ft}{\ms{F}_t}
\nc{\End}{\mathrm{End}}  
\nc{\Mod}{\mathrm{Mod}}  
\nc{\algD}{\mf{D}}
\nc{\g}{\mf{g}}
\nc{\gr}{\mathsf{gr}}
\nc{\U}{\mathsf{U}}
\nc{\E}{\mathsf{E}}

\nc{\CA}{\textrm{Circ}(\mc{A})}
\nc{\CB}{\textrm{Circ}(\mc{B})}
\nc{\usE}{\mathbb{E}}
\nc{\posB}{\textrm{pos}(\mc{B})}
\nc{\mE}{\mathcal{E}}
\nc{\ann}{\textrm{ann}}
\nc{\Span}{\textrm{Span}}

\newcommand{\mmod}{\text{-}\mathrm{mod}}

\usepackage{graphicx}    

\nc{\F}{F}
\nc{\M}{M}
\nc{\msf}{\mathsf}
\nc{\kt}{\mbf{k}}
\nc{\ko}{\mbf{k}(0)}
\nc{\Rep}{\textsf{Rep}}
\nc{\Ok}{\mc{O}_G \boxtimes \kt_X}
\nc{\Oko}{\mc{O}_G \boxtimes \ko_X}
\nc{\OYk}{\mc{O}_Y \boxtimes \kt_X}
\nc{\id}{\msf{id}}
\nc{\Pic}{\msf{Pic}}
\nc{\Xo}{\mf{X}}
\nc{\halfZ}{\frac{1}{2} \Z}
\nc{\txm}{\textrm}
\nc{\grr}{\mathrm{gr}}
\nc{\Spec}{\mathrm{Spec} \ }
\nc{\Frac}{\mathrm{Frac} \ }
\nc{\ra}{\longrightarrow}
\nc{\alg}{\mathrm{alg}}
\def\hp{\hphantom{x}}
\renewcommand{\o}{\otimes}
\nc{\Loc}{\mathsf{Loc}}
\nc{\Sec}{\mathsf{Sec}}

\usepackage{verbatim}  

\begin{document}

\title{On Deformation Quantizations of Hypertoric Varieties}
\author{Gwyn Bellamy}
\address{School of Mathematics and Statistics, University of Glasgow, University Gardens, Glasgow G12 8QW, Scotland}
\email{gwyn.bellamy@glasgow.ac.uk}
\author{Toshiro Kuwabara}
\address{Department of Mathematics,
National Research University Higher School of Economics,
7 Vavilova st., Moscow 117312 Russia}
\email{toshiro.kuwa@gmail.com}

\begin{abstract}
\noindent Based on a construction by Kashiwara and Rouquier, we present an analogue of the Beilinson-Bernstein localization theorem for hypertoric varieties. In this case, sheaves of differential operators are replaced by sheaves of $W$-algebras. As a special case, our result gives a localization theorem for rational Cherednik algebras associated to cyclic groups.
\end{abstract}

\maketitle

\section{Introduction}

Kontsevich \cite{Kontsevich} and Polesello and Schapira \cite{PS} have shown that one can construct a stack of ``$W$-algebroids'' (or deformation-quantization algebroids) on any symplectic manifold. These stacks of $W$-algebroids provide a quantization of the sheaf of holomorphic functions on the manifold. In certain cases, these stacks of $W$-algebroids are the algebroids associated to a sheaf of non-commutative algebras called $W$-algebras. Locally this is always the case. When the symplectic manifold in question is the Hamiltonian reduction of a space equipped with a genuine sheaf of $W$-algebras, Kashiwara and Rouquier \cite{KR} have shown that one can define a family of sheaves of $W$-algebras on the Hamiltonian reduction coming from the sheaf upstairs. This provides a large class of examples of sheaves of $W$-algebras on non-trivial symplectic manifolds. In this paper we study $W$-algebras on the simplest class of Hamiltonian reductions, those coming from the action of a torus $T$ on a symplectic vector space $V$. These spaces $Y(A,\delta)$, where $A$ is a matrix encoding the action of $T$ on $V$ and $\delta \in \X(T)$ is a character of $T$, are called hypertoric varieties. They were originally studied as hyperk\"ahler manifolds by Bielawski and Dancer \cite{BielawskiDancer}. Examples of hypertoric varieties include the cotangent space of projective $n$-space and resolutions of cyclic Kleinian singularities. More generally, the cotangent space of any smooth toric variety can be realized as a dense, open subvariety of the corresponding hypertoric variety.\\

One can also associate to the data of a reductive group $G$ acting on a symplectic vector space a certain family of non-commutative algebras $\U_\chi$, where $\chi \in \X(\mf{g})$ is a character of $\mf{g} = \mathrm{Lie} (G)$, called quantum Hamiltonian reductions. In the case $G = T$ is a torus, these algebras have been extensively studied by Musson and Van den Bergh \cite{MVdB}. The main goal of this paper is to prove a localization theorem, analogous to the celebrated Beilinson-Bernstein localization theorem \cite{BB}, giving an equivalence between the category of finitely generated modules for the quantum Hamiltonian reduction and a certain category of modules for a $W$-algebra. When the character $\delta$ is chosen to lie in the interior $C$ of a G.I.T. chamber in $\X(T)$ the hypertoric variety $Y(A,\delta)$ is a symplectic manifold. Then each character $\chi \in \X(\mf{t})$ gives a sheaf of $W$-algebras $\AW_\chi$ on $Y(A,\delta)$. Associated to $\AW_\chi$ is a category of ``good'' $\AW_\chi$-modules, $\Mod^{\mathrm{good}}_F(\AW_\chi)$ and a sub-category $\underline{\Mod}^{\mathrm{good}}_F(\AW_\chi)$ consisting (roughly) of those modules generated by their global section (the reader is referred to section 2 for the precise definition of these categories). Then we have natural localization and global section functors
\vspace*{2mm} 
$$
\begin{array}{ll}
\Loc : \U_\chi \mmod \longrightarrow \Mod^{\mathrm{good}}_F(\AW_\chi), & \Loc(M) = \AW_\chi \otimes_{\U_\chi} M, \\
 & \\
\Sec : \Mod^{\mathrm{good}}_F(\AW_\chi) \longrightarrow \U_\chi \mmod, & \Sec(\ms{M}) = \Hom_{\Mod_F^{\mathrm{good}}(\AW_\chi)}(\AW_\chi , \ms{M}).
\end{array}
$$
\vspace*{2mm} 
Our main result can be stated as follows.

\begin{thm}
Let $\chi \in C_{\Q}$. 
\begin{enumerate}[(i)]
\item The functor $\Loc$ defines an equivalence of categories 
$$
\U_\chi \mmod \, \stackrel{\sim}{\longrightarrow} \, \underline{\Mod}^{\mathrm{good}}_F(\AW_\chi) 
$$
with quasi-inverse $\Sec$.
\item There exists some $\theta \in C \cap \X(T)$ such that the functor $\Loc$ defines an equivalence of categories
$$
\U_{\chi + \theta} \mmod \, \stackrel{\sim}{\longrightarrow} \, \Mod^{\mathrm{good}}_F(\AW_{\chi + \theta}) 
$$
with quasi-inverse $\Sec$.
\end{enumerate}
\end{thm}

The theorem shows that localization always gives an equivalence of categories, provided one is sufficiently far away from the G.I.T. walls. A corollary of the above theorem is:

\begin{cor}
Let $\chi \in C_{\Q}$. If the global dimension of $\U_{\chi}$ is finite then the functor $\Loc$ defines an equivalence of categories 
$$
\U_\chi \mmod \, \stackrel{\sim}{\longrightarrow} \, \Mod^{\mathrm{good}}_F(\AW_\chi) 
$$
with quasi-inverse $\Sec$.
\end{cor}

A particular class of examples of hypertoric varieties are the minimal resolutions $\widetilde{\C^2 / \Z_m}$ of the Kleinian singularities of type $A$. Under mild restrictions on the parameters, the corresponding quantum Hamiltonian reductions are Morita equivalent to the rational Cherednik algebras $H_{\mbf{h}}$ associated to cyclic groups. Then a corollary of our main result is a localization theorem for these rational Cherednik algebras.

\begin{cor}
For $\mbf{h}$ not lying on a G.I.T. wall, the functor $\Loc( e \cdot - )$ defines an equivalence of categories 
$$
H_{\mbf{h}} \mmod \, \stackrel{\sim}{\longrightarrow} \, \Mod^{\mathrm{good}}_F(\AW_{\mbf{h}}) 
$$
with quasi-inverse $H_{\mbf{h}} e \otimes_{e H_{\mbf{h}} e} \Sec ( - )$.
\end{cor}

The content of each section is summarized as follows. In section 2 we introduce, following Kashiwara and Rouquier, $W$-algebras on symplectic manifolds in the equivariant setting. In section 3 we give a criterion for the $W$-affinity of a class of $W$-algebras on those symplectic manifolds that are obtained by Hamiltonian reduction of a vector space acted upon by a reductive group. The $W$-algebras on hypertoric varieties that we will consider later are a special case of this more general setup. The main result of this section is Theorem \ref{thm:equivalence}.\\

Hypertoric varieties are introduced in section 4 and we show that they posses the correct geometric properties that are required to apply the results of section 3. Using the results of Musson and Van den Bergh, we prove our main results, Theorem \ref{thm:hypertoricequivalence} and Corollary \ref{cor:finiteglobaldim}. In the final section we consider the special case where the hypertoric variety is the resolution of a Kleinian singularity of type $A$ and the global sections of the sheaf of $W$-algebras on this resolution can be identified with the spherical subalgebra of the rational Cherednik algebra associated to a cyclic group.     

\subsection{Convention}
Throughout, a \textit{variety} will always mean an integral, separated scheme of finite type over $\C$. A non-reduced space will be referred to as a scheme, again assumed to be over $\C$.

\section{$W$-algebras}\label{sub:Wdefn}

\subsection{}\label{subsection:W1} In this section we recall the definition of $W$-algebras as given in \cite{KR}. We state results about the existence and ``affinity'' of $W$-algebras. Let $X$ be a complex analytic manifold and let $\mc{O}_X$ denote the sheaf of regular, holomorphic functions on $X$. Denote by $\D_X$ the sheaf of differential operators on $X$ with holomorphic coefficients. Denote by $\mbf{k} = \C(( \hbar))$ the field of formal Laurent series in $\hbar$ and by $\mbf{k}(0)$ the subring $\C[[\hbar]]$ of formal functions on $\C$. Considering $\kt$ and $\ko$ as abelian groups, the corresponding sheaves of locally constant functions on $X$ will be denoted $\kt_X$ and $\ko_X$ respectively. Given $m \in \Z$, we define $\W_{T^* \C^n}(m)$ to be the sheaf of formal power series $\sum_{i \ge -m} \hbar^i a_i, \,  a_i \in \mc{O}_{T^* \C^n}$, on the cotangent bundle $T^* \C^n$ of $\C^n$. Let us fix coordinates $x_1, \ds, x_n$ on $\C^n$ and dual coordinates $\xi_1, \ds, \xi_n$ on $(\C^n)^*$, identifying $T^* \C^n$ with $\C^n \times (\C^n)^*$. Set $\W_{T^* \C^n} = \bigcup_{m \in \Z} \W_{T^* \C^n}(m)$. Then $\W_{T^* \C^n}$ is a sheaf of (non-commutative) $\mbf{k}$-algebras on $T^* \C^n$. Multiplication is defined by
\beq\label{eq:multdefn}
a \circ b = \sum_{\alpha \in \Z^n_{\ge 0}} \frac{\hbar^{| \alpha |}}{\alpha !} \p^{\alpha}_\xi a \cdot \p_{x}^\alpha b,
\eeq
where $| \alpha | = \sum_{i = 1}^n \alpha_i$, $\alpha ! = \alpha_1 ! \cdots \alpha_n !$ and $\p_{\xi}^\alpha = \frac{\p^{| \alpha |}}{\p^{\alpha_1} \xi_1 \cdots \p^{\alpha_n} \xi_n}$. There is a ring homomorphism $\D_{\C^n}(\C^n) \longrightarrow \W_{T^* \C^n}(T^* \C^n)$ given by $x_i \mapsto x_i$ and $\frac{\p}{\p x_i} \mapsto \hbar^{-1} \xi_i$. Note that $\W_{T^* \C^n}(0)$ is a $\ko$-subalgebra. We denote the symbol map for $\W_{T^* \C^n}$ by 
$$
\sigma_m \, : \, \W_{T^* \C^n}(m) \longrightarrow \W_{T^* \C^n}(m) / \W_{T^* \C^n}(m-1) \simeq \hbar^{-m} \mc{O}_{T^* \C^n}.
$$
The sheaf $\mc{O}_{T^* \C^n}$ is a sheaf of Poisson algebras with Poisson bracket given by 
$$
\{ x_i , x_j \} = \{ \xi_i , \xi_j \} = 0, \quad \{ \xi_i , x_j \} = \delta_{ij} \quad \forall \ i,j \in [1,n].
$$
One sees from equation (\ref{eq:multdefn}) that $\sigma_0 ([ a,b]) = \{ \sigma_0(a) , \sigma_0(b) \}$ for all $a, b \in \W_{T^* \C^n}(0)$.

\subsection{} Let us now assume that $X$ is a complex symplectic manifold with holomorphic $2$-form $\omega_X$. A map $f$ between open subsets $U \subset X$ and $V \subset Y$ of the symplectic manifolds $(X,\omega_1)$ and $(Y,\omega_2)$ is said to be a symplectic map if $f^* \omega_2 = \omega_1$. A symplectic map is always locally biholomorphic \cite[Lemma 5.5.2]{BjorkBook}, therefore by symplectic map we will actually mean a biholomorphic symplectic map. Based on \cite{Kontsevich} and \cite{PS}:

\begin{defn}\label{defn:Walgebra}
A \textit{$W$-algebra} on $X$ is a sheaf of $\mbf{k}$-algebras $\W_X$ together with a $\ko$-subalgebra $\W_X(0)$ such that for each point $x \in X$ there exists an open neighborhood $U$ of $x$ in $X$, a symplectic map $f \, : \, U \longrightarrow V \subset T^* \C^n$ and a $\mbf{k}$-algebra isomorphism $r \, : \, f^{-1} (\W_{T^* \C^n}|_V) \stackrel{~}{\longrightarrow} \W_X |_U$ such that:
\begin{enumerate}[(i)]
\item the isomorphism $r$ restricts to a $\ko$-isomorphism $f^{-1} (\W_{T^* \C^n}(0)|_V) \stackrel{~}{\longrightarrow} \W_X(0) |_U$,
\item setting $\W_X (m) = \hbar^{-m} \W(0)$ for all $m \in \Z$, we have $\sigma_0 \, : \, \W_X(0) \longrightarrow \W_X(0) / \W_X(-1) \simeq \mc{O}_X$ and the following diagram commutes
$$
\xymatrix{
f^{-1} (\W_{T^* \C^n}(0)|_V) \ar[rr]^r \ar[d]_{\sigma_0} & & \W_X(0) |_U \ar[d]^{\nu_0} \\
f^{-1} (\mc{O}_{T^* \C^n}) \ar[rr]_{f^{\sharp}} & & \mc{O}_X
}
$$
\end{enumerate}
\end{defn}

\subsection{} Note that the first statement of property $(ii)$ of Definition \ref{defn:Walgebra} is actually a consequence of property $(i)$. Definition \ref{defn:Walgebra} $(ii)$ implies that $\sigma_0 ([ a, b ]) = \{ \sigma_0(a) , \sigma_0(b) \}$ for all $a,b \in \W_X(0)$, where the Poisson bracket on $\mc{O}_X$ is the one induced from the symplectic form $\omega$ on $X$.

\subsection{Categories of $W$-modules} Unless explicitly stated, all modules will be left modules. Since $\W_X(0)$ is Noetherian (see \cite[(2.2.2)]{KR}), a $\W_X(0)$-module $\ms{M}$ is said to be \textit{coherent} if it is locally finitely generated. For a $\W_X$-module $\ms{M}$, a $\W_X (0)$-\textit{lattice} of $\ms{M}$ is a $\W_X(0)$-submodule $\ms{N}$ of $\ms{M}$ such that the natural map $\W \otimes_{\W (0)} \ms{N} \longrightarrow \ms{M}$ is an isomorphism. A $\W$-module $\ms{M}$ is said to be \textit{good} if for every relatively compact open set $U$ there exists a coherent $\W_X (0) |_{U}$-lattice for $\ms{M}|_{U}$. We will denote the category of left $\W_X$-modules as $\Mod(\W_X)$ and the full subcategory of good $\W_X$-modules as $\Mod^{\mathrm{good}}(\W_X)$. It is an abelian subcategory. If $\ms{M}(0)$ is a $\W_X(0)$-lattice of $\ms{M}$, set $\ms{M}(m) := \hbar^{-m} \ms{M}(0)$.

\begin{lem}\label{lem:eHc}
Let $\ms{M}$ be a coherent $\W_X$-module, equipped with a global $\W_X(0)$-lattice $\ms{M}(0)$. Then the filtration $\ms{M}(n)$, $n \in \Z$, is exhaustive, Hausdorff and complete; that is 
\begin{enumerate}[(i)]
\item $\bigcup_{n \in \Z} \ms{M}(n) = \ms{M}$;
\item $\bigcap_{n \in \Z} \ms{M}(n) = 0$; and
\item $\displaystyle \lim_{-\infty \leftarrow n} \ms{M} / \ms{M}(n) = \ms{M}$,
\end{enumerate}
where our terminology is chosen to agree with that of \cite[\S 5]{Weibel}.
\end{lem}

\begin{proof}
The statement (i) is true if $\ms{M} = \W_X$. But, by definition of a lattice, we have 
$$
\bigcup_{n \in \Z} \ms{M}(n) = \bigcup_{n \in \N} \W_X(n) \o_{\W_X(0)} \ms{M}(0) = \W_X \o_{\W_X(0)} \ms{M}(0) = \ms{M}.
$$ 
Part (ii) follows from \cite[Lemma 2.11]{KR}. Fix some open subset $U$ of $X$ and take a section 
$$
(f_n)_{n \in \Z} \in \lim_{-\infty \leftarrow n} (\ms{M} / \ms{M}(n))(U).
$$
Then, by part (i), there exists some integer $k > n$ such that the image $f_n$ of $f$ in $(\ms{M} / \ms{M}(n))(U)$ lies in $(\ms{M}(k) / \ms{M}(n))(U)$. Now by definition $f_n$ is the image of $f_{n-1}$ in the surjection $(\ms{M} / \ms{M}(n-1))(U) \ra (\ms{M} / \ms{M}(n))(U)$, hence $f_{n-1} \in (\ms{M}(k) / \ms{M}(n-1))(U)$ too. Thus $(\hbar^{-k} f_n)_{n \in \Z}$ is in $\displaystyle \lim_{-\infty \leftarrow n} (\ms{M}(0) / \ms{M}(n))(U)$. This implies that we have a surjective morphism 
$$
\kt_X \o_{\ko_X} \lim_{-\infty \leftarrow n} \ms{M}(0) / \ms{M}(n) \ra \lim_{-\infty \leftarrow n} \ms{M} / \ms{M}(n). 
$$
But it follows once again from \cite[Lemma 2.11]{KR} that 
$$
\ms{M} \simeq \kt_X \o_{\ko_X} \ms{M}(0) \simeq \kt_X \o_{\ko_X} \lim_{-\infty \leftarrow n} \ms{M}(0) / \ms{M}(n).
$$
Thus $\ms{M}$ surjects onto $\displaystyle \lim_{-\infty \leftarrow n} \ms{M} / \ms{M}(n)$. Part (ii) implies that this map is also injective. 
\end{proof}

\subsection{$G$-equivariance} Let $G$ be a complex Lie group acting symplectically on $X$; $T_g : X \stackrel{\sim}{\longrightarrow} X$, $\forall \ g \in G$. We assume that this action is Hamiltonian with moment map $\mu_X : X \longrightarrow \mf{g}^*$, where $\mf{g}$ is the Lie algebra of $G$.

\bdefn\label{defn:Gaction}
A $G$-action on the $W$-algebra $\W_X$ is a $\kt_X$-algebra isomorphism $\rho_g \, : \, T^{-1}_g \W_X \stackrel{\sim}{\longrightarrow} \W_X$ for every $g \in G$ such that $\rho_g(a)$ depends holomorphically on $g \in G$ for each section $a \in T^{-1}_g \W_X$ and $\rho_{g_1} \circ \rho_{g_2} = \rho_{g_1 g_2}$ for all $g_1,g_2 \in G$.
\edefn

\begin{defn}
Suppose we have fixed a $G$-action on $\W_X$. A quasi-$G$-equivariant $\W_X$-module is a left $\W_X$-module $\ms{M}$, together with a $\kt_X$-module isomorphism $\rho_g^{\ms{M}} \, : \, T^{-1}_g \ms{M} \stackrel{\sim}{\longrightarrow} \ms{M}$ for every $g \in G$ such that $\rho_g^{\ms{M}}(m)$ depends holomorphically on $g \in G$ for each section $m \in T^{-1}_g \ms{M}$, $\rho_{g}^{\ms{M}} \circ \rho_h^{\ms{M}} = \rho_{gh}^{\ms{M}}$ for all $g,h \in G$ and $\rho_{g}^{\ms{M}}(a \cdot m) = \rho_g(a) \cdot \rho_g^{\ms{M}}(m)$ for all $g \in G, a \in T^{-1}_g \W_X$ and $m \in T^{-1}_g \ms{M}$.  
\end{defn}

The category of quasi-$G$-equivariant $\W_X$-modules will be denoted $\Mod_G(\W_X)$. If $\ms{M}$ and $\ms{N}$ are elements in $\mathrm{Obj}(\Mod_G(\W_X))$ then a morphism $\phi \in \Hom_{\Mod_G(\W_X)}(\ms{M},\ms{N})$ is a collection of morphisms $\phi_U : \ms{M}(U) \longrightarrow \ms{N}(U)$ of $\W_X(U)$-modules, one for each open set $U \subset X$, which satisfies the usual conditions of being a $\W_X$-homomorphism, such that in addition, for each $g \in G$, the diagram 
$$
\xymatrix{
\ms{M}(T^{-1}_g(U)) \ar[rr]^{\phi_{T^{-1}_g(U)}} \ar[d]_{\rho^{\ms{M}}_g(U)} & & \ms{N}(T^{-1}_g(U)) \ar[d]^{\rho^{\ms{N}}_g(U)} \\
\ms{M}(U) \ar[rr]_{\phi_U} & & \ms{N}(U)
}
$$
is commutative. 

\bdefn\label{defn:momentmap}
Let $G$ act on the algebra $\W_X$. A map $\mu_{\W} : \mf{g} \longrightarrow \W_X(1)$ is said to be a \textit{quantized moment map} for the $G$-action if $\mu_\W$ satisfies the following properties:
\begin{enumerate}[(i)]
\item $[\mu_{\W}(A), a] = \frac{\textrm{d}}{\textrm{d}t} \rho_{\textrm{exp}(tA)}(a) |_{t = 0}$,
\item $\sigma_0(\hbar \mu_{\W}(A)) = \mu_X (A)$,
\item $\mu_{\W}(\textrm{Ad}(g) A) = \rho_g (\mu_{\W}(A))$,
\end{enumerate}
for every $A \in \mf{g}$, $a \in \W_X$ and $g \in G$.
\edefn

Let $\X(G) := \Hom_{\textrm{gp}}(G, {\mathbb C}^*)$ be the lattice of $G$-characters. Note that if $a \in \W_X$ is a $\theta$-semi-invariant of $G$ (that is, $\rho_g(a) = \theta(g) a, \, \forall g \in G$), where $\theta \in \X(G)$, then
\beq\label{eq:diffsemi}
[\mu_\W(A), a] = d \theta (A) a,
\eeq
where $d \, : \, \X(G) \longrightarrow (\mf{g}^*)^G$ is the differential sending a $G$-character to the corresponding $\mf{g}$-character. From now on we omit the symbol $d$ and think of $\theta \in \X(G)$ as a character for both $G$ and $\mf{g}$. For $\chi \in (\mf{g}^*)^G$ we set
\beq\label{eq:defineL}
\L_{X,\chi} = \W_{X} \bigm/ \sum_{A \in \mf{g}} \W_{X} (\mu_{\W}(A) - \chi(A)).
\eeq
Note that $\L_{X,\chi}$ is a good quasi-$G$-equivariant $\W_{X}$-module, and has lattice
$$
\L_{X,\chi}(0) := \W_{X}(0) \bigm/ \sum_{A \in \mf{g}} \W_{X}(-1) (\mu_{\W}(A) - \chi(A)).
$$
We will require the following result, whose proof is based on Holland's result \cite[Proposition 2.4]{Holland}.

\begin{prop}\label{lem:HollandProp}
Assume that the moment map $\mu_X$ is flat, then on $X$ we have an isomorphism of graded sheaves 
$$
\gr (\L_{X,\chi}) \simeq \bigoplus_{n \in \Z} \mc{O}_{\mu_X^{-1}(0)} \hbar^{-n}.
$$
\end{prop}

\begin{proof}
The moment map $\mu_{\W}$ makes $\W_X$ into a right $U(\mf{g})$-module. Let $\C_{\chi}$ be the one-dimensional $U(\mf{g})$-module defined by the character $\chi$ so that $\L_{X,\chi} = \W_X \o_{U(\mf{g})} \C_{\chi}$. As in \cite[Proposition 2.4]{Holland}, we denote by $B_{\bullet}$ the Chevalley-Eilenburg resolution of $\C_\chi$. Thus, $B_k = U(\mf{g}) \o \wedge^k \mf{g}$ and the differential is given by
\begin{align*}
d_k (f \o x_1 \wedge \dots \wedge x_k) & = \sum_{i = 1}^k (-1)^{i + 1} f (x_i - \chi(x_i)) \o x_1 \wedge \dots \wedge \hat{x}_i \wedge \dots \wedge x_k \\
 & \hp \hp + \sum_{1 \le i < j \le k} (-1)^{i + j} f \o [x_i , x_j] \wedge x_1 \wedge \dots \wedge \hat{x}_i \wedge \dots \wedge \hat{x}_j \wedge \dots \wedge x_k.
\end{align*}
Then $B_{\bullet}$ is a complex of free $U(\g)$-modules such that $H^0(B_{\bullet}) = \C_{\chi}$ and $H^k(B_{\bullet}) = 0$ otherwise. Let $C_{\bullet} = \W_X \o_{U(\mf{g})} B_{\bullet} = \W_X \o \wedge^{\bullet} \mf{g}$. The filtration on $\W_X$ induces a filtration $F_n C_k = \W_{X} (n - k) \o \wedge^k \mf{g}$ on the complex $C_{\bullet}$ such that $d_k (F_n C_k) \subseteq F_n C_{k-1}$ (recall that $\mu_{\W}(\mf{g}) \subset \W_{X}(1)$). Note that the filtration is not bounded above or below. However, by Lemma \ref{lem:eHc} the filtration on $C_{\bullet}$ is exhaustive, Hausdorff and complete. We denote by $E^r_{p,q}$ the spectral sequence corresponding to the filtration $F_n$ on $C_{\bullet}$. Since the filtration is exhaustive, Hausdorff and complete, the proof of \cite[Theorem 5.5.10]{Weibel} shows that the spectral sequence $E$ converges to $H_{\bullet}(C)$ (that the sequence is regular follows from the fact, to be shown below, that it collapses at $E^1$). By construction, we have an isomorphism of filtered sheaves $H^0(C) \simeq \L_{X,\chi}$ and hence $\gr (H_{0}(C)) \simeq \gr (\L_{X,\chi})$. Denote by $A$ the graded sheaf of algebras $\bigoplus_{n \in \Z} \mc{O}_X \hbar^{-n}$, where $\mc{O}_X$ is in degree zero and $\hbar$ has degree $-1$. The $0$-th page of the spectral sequence is given by 
$$
E_{p,q}^0 = \W_{X} (p - q) \o \wedge^{p+q} \mf{g} / \W_{X} (p - 1 - q) \o \wedge^{p+q} \mf{g} \simeq A_{p-q} \o \wedge^{p+q} \mf{g}.
$$
Since $\C[\mf{g}^*]$ is a domain and $\mu_X$ is assumed to be flat, $\mu_X^* : \mu_X^{-1} \mc{O}_{\g^*} \ra \mc{O}_X$ is an embedding and we may think of $\mu_X^{-1} \mc{O}_{\g^*}$ as a subsheaf of $\mc{O}_X$. Let $x_1, \ds, x_r$ be a basis of $\mf{g}$. Then \cite[Proposition 1.1.2]{CohMac} implies that $\hbar^{-1} x_1, \ds, \hbar^{-1} x_r$ form a regular sequence in $A$ at those points where they vanish. By Definition \ref{defn:momentmap} (2), the symbol $\sigma_1(\mu_{\W}(x_i))$ equals $\hbar^{-1} x_i \in A$. Thus the differential on $E^0$ is given by 
$$
d_{p+q} (f \o x_1 \wedge \dots \wedge x_{p+q}) = \sum_{i = 1}^{p+q} f \hbar^{-1} x_i \o x_1 \wedge \dots \wedge \hat{x}_i  \wedge \dots \wedge x_{p+q}.
$$
As is explain in \cite[Proposition 2.4]{Holland}, the only non-zero homology of $E^0$ is in the $(p,-p)$ position where we have
$$
E^1_{(p,-p)} = \frac{A_p}{A_{p-1} \cdot \hbar \mu^*_X(\mf{g})} \simeq \mc{O}_{\mu_X^{-1}(0)} \hbar^{-p}.
$$
Therefore the sequence collapses at $E^1$ and we have 
$$
\gr (\L_{X,\chi})_p \simeq \gr (H_{0}(C))_p \simeq \mc{O}_{\mu_X^{-1}(0)} \hbar^{-p}
$$
as required. 
\end{proof}

\subsection{$F$-actions} Here we repeat the definition of $F$-action on $\W_X$-modules as defined in \cite[\S 2.3.1]{KR}. Let $\C^{\times} \ni t \mapsto T_t \in \mathrm{Aut}(X)$ denote an action of the torus $\C^{\times}$ on $X$ such that the symplectic $2$-form is a semi-invariant of positive weight, $T_t^* \omega_X = t^m \omega_X$ for some $m > 0$.

\bdefn
An $F$-action with exponent $m$ on $\W_X$ is an action of the group $\C^{\times}$ on $\W_X$ as in Definition \ref{defn:Gaction} except that $\C^{\times}$ also acts on $\hbar$: if $\ms{F}_t \, : \, T^{-1}_t \W_X \stackrel{\sim}{\longrightarrow} \W_X$  denotes the action of $t \in \C^{\times}$ then we require that $\ms{F}_t (\hbar) = t^m \hbar$ for all $t \in \C^{\times}$.
\edefn

It will be convenient to extend the $F$-action of $\C^{\times}$ to an action on $\W [\hbar^{1/m}] := \kt(\hbar^{1/m}) \otimes_{\kt} \W$ by setting $\ms{F}_t(\hbar^{1/m}) = t \hbar^{1/m}$. The category of $F$-equivariant $\W_X$-modules will be denoted $\Mod_F(\W_X)$. As noted in \cite[\S 2.3.1]{KR}, $\Mod_F(\W_X)$ is an abelian category. It is also noted in \cite[\S 2.3]{KR} that if there exists a relatively compact open subset $U$ of $X$ such that $\C^{\times} \cdot U = X$ then every good, $F$-equivariant $\W_X$-module admits globally a coherent $\W_X(0)$-lattice. Such an open set $U$ will exist in the cases we consider. The following lemma will be used later.

\begin{lem}\label{lem:invglobalcyclic}
Let $\ms{M},\ms{N} \in \Mod_{F,G}^{\mathrm{good}}(\W_X)$. Assume that $\ms{M} \simeq \W_X / \mathscr{I}$ is a cyclic $\W_X$-module, generated by some $G,F$-invariant element, where $\mathscr{I}$ is a left ideal generated by finitely many global section. Then $\Hom_{\Mod_{F,G}^{\mathrm{good}}(\W_X)}(\ms{M},\ms{N}) = \Hom_{\W_X(X)}(\ms{M}(X),\ms{N}(X))^{G,F}$.
\end{lem}

\subsection{Example}

Let $V$ be an $n$-dimensional vector space. We fix $X = T^* V$ with co-ordinates $x_1, \ds, x_n, \xi_1, \ds ,\xi_n$ and define an action $T_t$ of $\C^{\times}$ on $X$ such that the corresponding action on co-ordinate functions is given by $T_t( x_i) = t x_i$ and $T_t (\xi_i) = t \xi_i$. Then $T_t^* \omega_X = t^2 \omega_X$. We extend this to an $F$-action on $\W_{T^* V}$ by setting $\Ft (\hbar) = t^2 \hbar$. Let $\algD (V)$ denote the ring of algebraic differential operators on $V$.

\blem\label{lem:Finv}
Taking $F$-invariants in $\W_{T^* V}(T^* V)$ gives
\begin{align*}
\End_{\Mod_{F}(\W_{T^* V}[\hbar^{1/2}])}(\W_{T^* V}[\hbar^{1/2}])^{\txm{opp}} &= \C[\hbar^{-1/2}x_i,\hbar^{-1/2}\xi_i \, : \, i \in [1,n] ]\\
 & = \C \left[ \hbar^{-1/2}x_i,\hbar^{1/2} \frac{\p}{\p x_i} \, : \, i \in [1,n] \right],
\end{align*}
where the second equality comes from $\algD (V) \hookrightarrow \W_{T^* V}({T^* V})$, $x_i \mapsto x_i$ and $\frac{\p}{\p x_i} \mapsto \hbar^{-1} \xi_i$.
\elem

\begin{proof}
We can identify $\End_{\Mod_{F}(\W_{T^* V}[\hbar^{1/2}])}(\W_{T^* V}[\hbar^{1/2}])^{\txm{opp}}$ with the algebra $\W_{T^* V}[\hbar^{1/2}](T^* V)^{F}$ of $F$-invariant global sections. Since ${T^* V}$ is connected, taking power series expansion in a sufficiently small neighborhood of $0 \in {T^* V}$ defines an embedding $\mc{O}_{T^* V}({T^* V}) \hookrightarrow \C[[x_1, \ds , x_n , \xi_1, \ds , \xi_n]]$. As $\C^{\times}$-modules we can identify $\W_{T^* V}[\hbar^{1/2}]$ with $\mc{O}_{T^* V} \widehat{\otimes} \C((\hbar^{1/2}))$ and we get a $\C^{\times}$-equivariant embedding
$$
\W_{T^* V}[\hbar^{1/2}]({T^* V}) \hookrightarrow \C[[x_1, \ds , x_n , \xi_1, \ds , \xi_n]] \widehat{\otimes} \C((\hbar^{1/2})),
$$
where we denote by $\widehat{\otimes}$ the completed tensor product with respect to the linear topology.
Taking invariants gives the desired result.
\end{proof}

A trivial application of Theorem \ref{thm:equivalence} below, with $f = \textrm{id}_{\C^n}$ and $G = \{ 1 \}$, shows that 
$$
\Mod_F(\W_{T^* V}[\hbar^{1/2}]) \simeq \C \left[ \hbar^{-1/2}x_i,\hbar^{1/2} \frac{\p}{\p x_i} \, : \, i \in [1,n] \right] \mmod.
$$

\section{$W$-Affinity}\label{sec:mainaffinity}

In this section we give a criterion for the $W$-affinity of a class of $W$-algebras on those symplectic manifolds that are obtained by Hamiltonian reduction. 

\subsection{The geometric setup}\label{sec:affinity} Let $V$ be an $n$-dimensional vector space over $\C$. Its cotangent bundle $T^* V$ has the structure of a complex symplectic manifold. Let $G$ be a connected, reductive algebraic group acting algebraically on $V$. This action induces a Hamiltonian action on $T^* V$ and we have a moment map
$$
\mu_{T^* V} : T^* V \longrightarrow \g^* := (\textrm{Lie } G)^*
$$
such that $\mu_{T^* V}(0) = 0$. We fix a character $\vartheta \in \X(G)$. Let $\Xo$ be the open subset of all $\vartheta$-semi-stable points in $T^* V$ and denote the restriction of $\mu_{T^* V}$ to $\Xo$ by $\mu_{\Xo}$. We assume that the following holds:
\begin{align*}
\textrm{(i) } & \textrm{ the set $\mu_{\Xo}^{-1}(0)$ is non-empty}, \\ 
\textrm{(ii) } & \textrm{ $G$ acts freely on $\mu_{\Xo}^{-1}(0)$}, \\
\textrm{(iii) } & \textrm{ the moment map $\mu_{T^* V}$ is flat}.
\end{align*}
Denote 
$$
Y_{\vartheta} := \mu^{-1}_{\Xo}(0) /\!/ G = \mathrm{Proj} \ \bigoplus_{n \ge 0} \C[\mu^{-1}_{T^* V}(0)]^{n \vartheta}
$$
and write 
\[
f: Y_{\vartheta} \longrightarrow \mu^{-1}_{T^* V}(0) /\!/ G =: Y_0
\]
for the corresponding projective morphism. Condition (i) implies that the categorical quotient $Y_{\vartheta}$ is non-empty. Condition (ii) implies that the morphism $\mu_{\Xo}$ is regular at all points in $\mu^{-1}_{\Xo}(0)$ and hence $Y_{\vartheta}$ is a non-singular symplectic manifold. Condition (iii) will be used in Proposition \ref{lem:qhriso}. We add to our previous assumptions,
$$
\textrm{(iv) } \ \textrm{the morphism $f$ is birational and $Y_0$ is a normal variety.}
$$
In the case of hypertoric varieties, it is shown in section \ref{sec:geometry} that assumptions (i)-(iv) hold when the matrix $A$ is unimodular. 

\begin{lem}\label{lem:globalsection}
Let $\mc{O}_{Y_{\vartheta}}^{\mathrm{alg}}$, resp. $\mc{O}_{Y_0}^{\mathrm{alg}}$ denote the sheaf of regular functions on $Y_{\vartheta}$, resp. $Y_0$. If $Y_{\vartheta},Y_0,f$ satisfy assumption (iv) then $\Gamma(Y_{\vartheta},\mc{O}_{Y_{\vartheta}}^{\mathrm{alg}}) = \Gamma(Y_0,\mc{O}_{Y_0}^{\mathrm{alg}})$. 
\end{lem}

\begin{proof}
It is well-known that the condition implies the statement of the lemma, but we were unable to find any suitable reference therefore we include a proof for the readers convenience. For $s \ge 0$, fix $R_s = \C[\mu^{-1}_{T^* V}(0)]^{s \vartheta}$ and $R = \oplus_{s \ge 0} R_s$ so that $Y_{\vartheta} = \mathrm{Proj} \ R$ and recall that $f$ is the canonical projective morphism from $Y_{\vartheta}$ to $Y_0$. By Hilbert's Theorem (see \cite[Zusatz 3.2]{Kraft}), $R$ is finitely generated as an $R_0$-algebra. Let $x_1, \dots ,x_n \in R$ be homogeneous generators (of degree at least one) of $R$ as an $R_0$-algebra. Then the affine open sets $D_+(x_i)=\Spec R_{(x_i)}$ form an open cover of $Y_{\vartheta}$ and   
$$
\Gamma(Y_{\vartheta},\mc{O}_{Y_{\vartheta}}^{\mathrm{alg}}) = \bigcap_{i = 1}^n R_{(x_i)} \subseteq \bigcap_{i = 1}^n R_{x_i},
$$
Let $r \in \Gamma(Y_{\vartheta}, \mathcal{O}_{Y_{\vartheta}}^{\mathrm{alg}})$. Then, for each $i$, there exists an $m$ such that $x_i^m \cdot r \in R$. We choose one $m$ sufficiently large so that $x_i^m \cdot r \in R$ for all $i$. Since the $x_i$ generate $R$, we actually have $y \cdot r \in R$ for all $y \in R_s$ and $s \geq m_0 := nmd$, where $d$ is the maximum of the degrees of $x_1, \ds, x_n$. Therefore $y \cdot r \in R$ for all $y \in \bigoplus_{s \ge m_0} R_s$. Since $r$ has degree zero, $y \cdot r \in \bigoplus_{s \ge m_0} R_s$ for all $y \in \bigoplus_{s \ge m_0} R_s$. Inductively, $y \cdot r^q \in \bigoplus_{s \ge m_0} R_s$ for all $q \ge 1$. Take $y = x_1^{m_0}$, then $r^q \in \frac{1}{x_1^{m_0}} R$ for all $q \ge 1$ and hence $R[r] \subset \frac{1}{x_1^{m_0}} R$. But, by Hilbert's basis theorem, $R$ is Noetherian and the $R$-module $\frac{1}{x_1^{m_0}} R$ is finitely generated, hence the algebra $R[r]$ is finite over $R$. This means $r$ satisfies some monic polynomial $u^t + r_1 u^{t-1} + \dots + r_t$ with coefficients in $R$. However $R$ has degree zero so without loss of generality $r_i \in R_0$. Thus $r$ is in the integral closure of $R_0$ in the degree zero part of the field of fractions of $R$. Now \cite[Theorem 7.17]{Hartshorne} says that, since the map $f$ is projective and birational, there exists an ideal $I$ in $R_0$ such that $R_k \simeq I^k$ as $R_0$-modules and we have an isomorphism of graded rings $R \simeq \bigoplus_{k \ge 0} I^k$. That is, $Y_{\vartheta}$ is isomorphic to the blowup of $Y_0$ along $V(I)$. Therefore we can identify the degree zero part of the field of fractions of $R$ with the field of fractions of $R_0$. Since $R_0$ is assumed to be normal, $r \in R_0$ as required. 
\end{proof}

\subsection{} 
The quotient morphism will be written $p \, : \, \mu^{-1}_{\Xo}(0) \longrightarrow Y_{\vartheta}$. For each character $\theta \in \X(G)$ and vector space $M$ on which $G$ acts, we denote by $M^{\theta}$ the set of element $m \in M$ such that $g \cdot m = \theta(g) m$ for all $g \in G$. We can define a coherent sheaf $L_{\theta}$ on the quotient $Y_{\vartheta}$ by
$$
L_{\theta}(U) := \left[ \mc{O}_{\mu^{-1}_{\Xo}(0)}(p^{-1}(U)) \right]^{\theta}.
$$
Since $G$ acts freely on $\mu^{-1}_{\Xo}(0)$, $L_{\theta}$ is a line bundle on $Y_{\vartheta}$. 

\subsection{Quantum Hamiltonian reduction}\label{sec:MQHR}
Differentiating the action of $G$ on $V$ produces a morphism of Lie algebras $\mu_D \, : \, \mf{g} \longrightarrow \textrm{Vect}(V)$, from $\mf{g}$ into the Lie algebra of algebraic vector fields on $V$:
$$
\mu_D(A)(r) := \frac{\textrm{d}}{\textrm{d}t} a^*_{\textrm{exp}(tA)}(r) |_{t = 0},
$$
where $a \, : \, G \times V \longrightarrow V$ is the action map and $a^* \, : \, G \times \mc{O}(V) \longrightarrow \mc{O}(V)$ the induced action on functions. We write $\algD(V)$ for the ring of \textit{algebraic} differential operators on $V$. Since $\textrm{Vect}(V) \subset \algD(V)$ we get a map $\mu_D : \mf{g} \longrightarrow \algD(V)$ which extends to an algebra morphism $U(\mf{g}) \longrightarrow \algD(V)$. For $\chi \in (\g^*)^G$, $\theta \in \mathbb{X}$, we define the left $\algD(V)$-module
\[
\L_{D, \chi} := \algD(V) \bigm/ \sum_{A \in \g} \algD(V)
(\mu_D(A) - \chi(A)),
\]
and the algebra, respectively $(\U_\chi, \U_{\chi + \theta})$-bimodule,
\[
\U_\chi = \left(\End_{\algD(V)}(\L_{D,\chi})^G \right)^{\textrm{opp}}, \quad
\U_\chi^{\theta} = \Hom_{\algD(V)}(\L_{D, \chi},
\L_{D, \chi+ \theta} \otimes \C_{\theta})^G.
\]
Fix $\chi \in (\g^*)^G$ and $\theta \in \mathbb{X}$. We consider the following natural homomorphisms:
\beq\label{eq:timesone}
\U_{\chi + \theta}^{-\theta} \otimes_{\C}
\U_{\chi}^{\theta} \longrightarrow
\U_{\chi+\theta}, \quad \phi \otimes \psi \mapsto (\msf{id}_{\U_{\chi + \theta}} \otimes \msf{ev}) \circ (\psi \otimes \msf{id}_{-\theta}) \circ \phi,
\eeq
\beq\label{eq:timestwo}
\U_{\chi}^{\theta} \otimes_{\C}
\U_{\chi + \theta}^{-\theta} \longrightarrow
\U_{\chi}, \quad \phi \otimes \psi \mapsto (\msf{id}_{\U_{\chi}} \otimes \msf{ev}) \circ (\psi \otimes \msf{id}_{\theta}) \circ \phi,
\eeq
where $\circ$ is composition of morphisms and $\msf{ev} \, : \, \C_{- \theta} \otimes \C_\theta \longrightarrow \C$ is the natural map. We write $\chi \rightarrow \chi + \theta$ if the map (\ref{eq:timesone}) is surjective and similarly $\chi + \theta \rightarrow \chi$ if the map (\ref{eq:timestwo}) is surjective. Note that if $\chi + \theta \leftrightarrows \chi$ then, as shown in \cite[Corollary 3.5.4]{MR}, the categories $ \U_{\chi} \mmod$ and $ \U_{\chi + \theta} \mmod$ are Morita equivalent. 

\subsection{The sheaf of $W$-algebras}\label{subsection:defineqmm} Denote by $\W_{\Xo}$ the restriction of the canonical $W$-algebra $\W_{T^* V}$ to $\Xo$. We define an action of the torus $\C^{\times}$ on $T^* V$ by $T_t(v) = t^{-1}v$ for all $v \in T^*V$; $\Xo$ is a $\C^{\times}$-stable open set. The algebra $\W_{\Xo}$ is then equipped with an $F$-action of weight $2$ as defined in the setup of Lemma \ref{lem:Finv}. Define $\widetilde{\W}_{T^* V} := \W_{T^* V} [\hbar^{1/2}]$ and write $\widetilde{\W}_{\Xo}$ for its restriction to $\Xo$. As noted in (\ref{subsection:W1}), we have an embedding $j : \algD (V) \hookrightarrow \W_{T^* V}$, $x_i \mapsto x_i$ and $\partial / \partial x_i \mapsto \hbar^{-1} \zeta_i$. Composing this morphism with the map $\mu_D \, : \, \mf{g} \longrightarrow \algD(V)$ gives us a map $\mu_{\W} = j \circ \mu_{D} \, : \, \mf{g} \longrightarrow \W_{T^* V}$. It is a quantized moment map in the sense of Definition \ref{defn:momentmap}. Then, as in (\ref{eq:defineL}), for each $\chi \in (\mf{g}^*)^G$, we have defined the $\widetilde{\W}_{T^* V}$-module $\L_{T^* V,\chi}$. Its restriction to $\Xo$ is denoted $\L_\chi$. Recall that $\L_\chi$ is a good quasi-$G$-equivariant $\widetilde{\W}_{\Xo}$-module. If we let $\C^{\times}$ act trivially on $\mf{g}$ then the morphism $\mu_{\W}$ is $F$-equivariant and hence $\L_\chi$ is equipped with an $F$-action. The image of $1$ in $\L_\chi$ will be denoted by $u_\chi$.

\subsection{} 
In \cite{KR}, Kashiwara and Rouquier show that one can quantize the process of Hamiltonian reduction to get a family of sheaves of $W$-algebras on $Y_{\vartheta}$ beginning from a $W$-algebra on $T^* V$. Set
$$
\AW_\chi = ((p_* \ms{E}nd_{\widetilde{\W}_{\Xo}}(\L_\chi))^G)^{\textrm{opp}} \quad  \textrm{and} \quad \AW_{\chi, \theta} = (p_* \ms{H}om_{\widetilde{\W}_{\Xo}}( \L_\chi, \L_{\chi + \theta} \otimes \C_{\theta}))^G,
$$
where $\theta \in \X(G)$ and $\C_\theta$ denotes the corresponding one dimensional $G$-module. By \cite[Proposition 2.8]{KR}, $\AW_\chi$ is a $W$-algebra on $Y_{\vartheta}$ and $\AW_{\chi, \theta}$ is a $(\AW_{\chi}, \AW_{\chi + \theta})$-bimodule. Let
$$
\AW_\chi (0) = ((p_* \ms{E}nd_{\widetilde{\W}_{\Xo}(0)}(\L_\chi(0)))^G)^{\textrm{opp}} \, \textrm{  and } \, \AW_{\chi, \theta}(0) = (p_* \ms{H}om_{\widetilde{\W}_{\Xo}(0)}( \L_\chi(0), \L_{\chi + \theta}(0) \otimes \C_\theta))^G,
$$
so that $\AW_{\chi, \theta}(0)$ is a $\AW_\chi (0)$-lattice of $\AW_{\chi, \theta}$. We have $\AW_\chi(0) / \AW_\chi(-1/2) \simeq \O_{Y_{\vartheta}}$ and, as noted in \cite[Proposition 2.8 (iii)]{KR}, $\AW_{\chi, \theta}(0) / \AW_{\chi, \theta}(-1/2) \simeq L_{- \theta}$ where $L_{\theta}$ is the line bundle as defined above. We say that a good $\AW_\chi$-module $\ms{M}$ is generated, locally on $Y_0$, by its global sections if for each $y \in Y_0$ there exists some open neighborhood (in the complex analytic topology) $U \subset Y_0$ of $y$ such that the natural map of left $(\AW_\chi)_{| \, f^{-1}(U)}$-modules $(\AW_\chi)_{| \, f^{-1}(U)} \otimes \ms{M}(f^{-1}(U)) \longrightarrow \ms{M}_{| \, f^{-1}(U)}$ is surjective.

\begin{defn}\label{defn:underlinemod}
We denote by $\underline{\Mod}^{\mathrm{good}}_F(\AW_\chi)$ the full subcategory of $\Mod^{\mathrm{good}}_F(\AW_\chi)$ consisting of all good, $F$-equivariant $\AW_\chi$-modules $\ms{M}$ such that
\begin{enumerate}[(i)]
\item $\ms{M}$ is generated, locally on $Y_0$, by its global sections.
\item For any non-zero submodule $\ms{N}$ of $\ms{M}$ in $\Mod^{\mathrm{good}}_F(\AW_\chi)$ we have $\Hom_{\Mod^{\mathrm{good}}_F (\AW_\chi)}(\AW_\chi, \ms{N}) \neq 0$.
\end{enumerate}
\end{defn}

\subsection{$W$-affinity}

We can now state the main result relating the sheaf of $W$-algebras $\AW_\chi$ on $Y_{\vartheta}$ and the algebra of quantum Hamiltonian reduction $\U_\chi$. 

\bthm\label{thm:equivalence}
Let $\AW_\chi$ and $\U_\chi$ be as above and choose some $\theta \in \X(G)$ such that $L_{\theta}$ is ample.
\begin{enumerate}[(i)]
\item There is an isomorphism of algebras $\Gamma(Y_{\vartheta},\AW_\chi)^F \simeq \U_\chi$.
\item Assume that we have $\chi \leftarrow \chi + n \theta$ for all $n \in \Z_{\ge 0}$. Then the functor $\ms{M} \mapsto \Hom_{\Mod^{\mathrm{good}}_F (\AW_\chi)}(\AW_\chi, \ms{M})$ defines an equivalence of categories
$$
\underline{\Mod}^{\mathrm{good}}_F (\AW_\chi) \stackrel{\sim}{\longrightarrow} \U_\chi \mmod
$$
with quasi-inverse $\ms{M} \mapsto \AW_\chi \otimes_{\U_\chi} \ms{M}$.
\item Assume that we have $\chi \leftrightarrows \chi + n \theta$ for all $n \in \Z_{\ge 0}$. Then the functor $\ms{M} \mapsto \Hom_{\Mod^{\mathrm{good}}_F (\AW_\chi)}(\AW_\chi, \ms{M})$ defines an equivalence of categories
$$
\Mod^{\mathrm{good}}_F (\AW_\chi) \stackrel{\sim}{\longrightarrow} \U_\chi \mmod
$$
with quasi-inverse $\ms{M} \mapsto \AW_\chi \otimes_{\U_\chi} \ms{M}$.
\end{enumerate}
\ethm

The proof of Theorem \ref{thm:equivalence} will occupy the remainder of section \ref{sec:mainaffinity}. 

\subsection{Proof of the theorem}\label{sec:defineA}

We fix $\AW_\chi$, $\U_\chi$ and $L_{\theta}$ as in Theorem \ref{thm:equivalence}. First we require some preparatory lemmata. Denote by $\iota$ the embedding $\algD (V) \hookrightarrow \widetilde{\W}_{T^* V}(T^* V)$ given by $x_i \mapsto \hbar^{-1/2}x_i$ and $\p_i \mapsto \hbar^{-1/2} \xi_i$. Equip $\algD(V)$ with a $\halfZ$-filtration $F_{\bullet} \ \algD(V)$ by placing $x_i$ and $\p_i$ in degree $1/2$ (this is the Bernstein filtration). Then $\iota$ is a strictly filtered embedding in the sense that 
$$
\iota ( F_k \algD(V)) = \iota (\algD(V)) \cap \widetilde{\W}_{T^* V}(T^* V)(k), \quad \forall \ k \in \halfZ.
$$
By Lemma \ref{lem:Finv}, the image of $\algD (V)$ in $\widetilde{\W}_{T^* V}(T^* V)$ is $\widetilde{\W}_{T^* V}(T^* V)^F$. This implies, since $\C^{\times}$ is reductive and $\mu_{\W}$ is equivariant, that 
\begin{align}\label{eq:idealintersection}
\sum_{A \in \mf{g}} \widetilde{\W}_{T^* V}(T^* V) (\mu_{\W}(A) - \chi(A)) \cap \iota (\algD (V)) & = \sum_{A \in \mf{g}} \widetilde{\W}_{T^* V}(T^* V)^F (\mu_{\W}(A) - \chi(A)) \\
 & = \sum_{A \in \mf{g}} \algD (V) (\mu_{D}(A) - \chi(A)).
\end{align}

\begin{lem}\label{lem:isoone}
As shown in \cite[Lemma 2.2]{GGS}:
\begin{enumerate}[(i)]
\item Multiplication in $\algD(V)$ defines an algebra structure on $( \L_{D, \chi})^G$ such that there is isomorphism of algebras $\U_\chi \stackrel{\sim}{\longrightarrow} ( \L_{D, \chi})^G$ given by $\phi \mapsto \phi( u_\chi )$ with inverse $f \mapsto r_f$, where $r_f = \cdot f$ is right multiplication by $f$.
\item We have an isomorphism of $(\U_\chi$, $\U_{\chi + \theta})$-bimodules $\U_{\chi}^{\theta} \stackrel{\sim}{\longrightarrow} (\L_{D, \chi + \theta})^{- \theta}$ given by $\phi \mapsto f$, where $\phi( u_{\chi}) = f u_{\chi+\theta} \otimes \theta$, with inverse  $f u_{\chi+\theta} \mapsto r_f \otimes \theta$.
\end{enumerate}
\end{lem}

Let us introduce
$$
\E_{\chi} = (\End_{\Mod_{G,F} \widetilde{\W}_{\Xo}}( \L_{\chi}))^{\textrm{opp}} \qquad \text{ and } \qquad \E^{\theta}_\chi = \Hom_{\Mod_{F,G} \widetilde{\W}_{\Xo}}(\L_\chi, \L_{\chi + \theta} \otimes \C_{\theta}),
$$
so that $\E^{\theta}_\chi$ is a $(\E_\chi,\E_{\chi + \theta})$-bimodule and $\L_\chi$ is a $(\widetilde{\W}_{\Xo}, \E_\chi)$-bimodule. By Lemma \ref{lem:invglobalcyclic}, we can identify 
$$
\E_{\chi} = (\End_{\widetilde{\W}_{\Xo}(\Xo)}(\L_{\chi})^{G,F})^{\textrm{opp}} \qquad \text{ and } \qquad \E^{\theta}_\chi = \Hom_{\widetilde{\W}_{\Xo}(\Xo)}(\L_\chi, \L_{\chi + \theta} \otimes \C_{\theta})^{G,F}.
$$
Note that equation (\ref{eq:idealintersection}) implies that the map $\iota$ induces an embedding $\iota : \L_{D, \chi} \hookrightarrow \L_{T^* V, \chi}(T^* V)$ and after taking $G,F$-invariants 
\beq\label{eq:someequality}
\iota : \U_{\chi} \stackrel{\sim}{\longrightarrow} (\End_{\widetilde{\W}_{T^* V}(T^* V)}(\L_{T^* V,\chi})^{G,F})^{\textrm{opp}}
\eeq
and
$$
\U^{\theta}_\chi \simeq \Hom_{\widetilde{\W}_{T^* V}(T^* V)}(\L_{T^* V,\chi}, \L_{T^* V,\chi + \theta} \otimes \C_{\theta})^{G,F}.
$$

\bprop\label{lem:qhriso}
We have a filtered isomorphism $\mbf{\Psi}_\chi : \U_{\chi} \stackrel{\sim}{\longrightarrow} \E_{\chi}$ in the sense that 
$$
\mbf{\Psi}_\chi (F_k \U_{\chi}) = F_k  \E_{\chi}, \quad \forall \ k \in \Z.
$$
\eprop

\begin{proof}
The isomorphism (\ref{eq:someequality}) induced by the embedding $\iota$ is filtered in the same sense as $\mbf{\Psi}_\chi$ above. Therefore it suffice to show that the natural map 
$$
(\L_{T^* V,\chi}(T^* V))^{G,F} = (\End_{\widetilde{\W}_{T^* V}(T^* V)}(\L_{T^* V,\chi})^{G,F})^{\textrm{opp}} \longrightarrow (\End_{\widetilde{\W}_{\Xo}(\Xo)}(\L_{\Xo,\chi})^{G,F})^{\textrm{opp}} = (\L_{\Xo,\chi}(\Xo))^{G,F}
$$
is a filtered isomorphism. The morphism of localization $\L_{T^* V,\chi}(T^* V) \longrightarrow \L_{T^* V,\chi}(\Xo)$ is clearly filtered in the weaker sense that it restricts to a map $\L_{T^* V,\chi}(T^* V)(k) \longrightarrow \L_{T^* V,\chi}(\Xo)(k)$ for each $k \in \halfZ$. Since the moment map $\mu_{T^* V}$ is assumed to be flat, Proposition \ref{lem:HollandProp} says that the morphism of associated graded spaces is the natural localization map
$$
\bigoplus_{k \in \halfZ} \mc{O}_{\mu_{T^* V}^{-1}(0)}(T^* V) \hbar^{-k} \ra \bigoplus_{k \in \halfZ} \mc{O}_{\mu_{T^* V}^{-1}(0)}(\Xo) \hbar^{-k}.
$$
Note that the filtration on $\L_{T^* V,\chi}$ is stable with respect to both $G$ and $F$. Lemma \ref{lem:eHc} says that the globally defined good filtration on $\L_{T^* V,\chi}$ is exhaustive and Hausdorff. Therefore, taking invariants with respect to $G$ and $F$, it suffices to show that 
$$
\bigoplus_{k \in \halfZ} \left( \mc{O}_{\mu_{T^* V}^{-1}(0)} (T^* V) \hbar^{-k} \right)^{G,F} \ra \bigoplus_{k \in \halfZ} \left( \mc{O}_{\mu_{T^* V}^{-1}(0)}(\Xo) \hbar^{-k} \right)^{G,F}
$$
is an isomorphism. But, since the $F$-action is contracting, 
$$
(\mc{O}_{\mu_{T^* V}^{-1}(0)} (T^* V) \hbar^{-k})^{G,F} = \C[\mu_{T^* V}^{-1}(0)]_{-2k}^{G}
$$
which is the space of $G$-invariant homogeneous polynomials on $\mu_{T^* V}^{-1}(0)$ of degree $-2k$. Similarly, 
$$
(\mc{O}_{\mu_{(T^* V)}^{-1}(0)} (\Xo) \hbar^{-k})^{G,F} = \C[\mu_{\Xo}^{-1}(0)]_{-2k}^{G}.
$$
Therefore the result follow from Lemma \ref{lem:globalsection} which says that 
$$
\C[\mu_{\Xo}^{-1}(0)]^{G} = \Gamma(Y_{\vartheta},\mc{O}_{Y_{\vartheta}}^{\alg}) = \Gamma(Y_0,\mc{O}_{Y_0}^{\alg}) = \C[\mu_{T^* V}^{-1}(0)]^{G}. 
$$
\end{proof}

\begin{remark}
In general, it is \textit{not} true that $\U_\chi^\theta \simeq \E_\chi^\theta$ when $\theta \neq 0$.
\end{remark}

\subsection{Shifting} 

The localization theorem relies on the following result by Kashiwara and Rouquier:

\bthm[Theorem 2.9, \cite{KR}]\label{thm:KRConditions}
Let $\AW_{\chi, \theta}(0)$ and $L_{\theta}$ be as above such that $L_{\theta}$ is ample.
\begin{enumerate}[(i)]
\item Assume that for all $n \gg 0$ there exists a finite dimensional vector space $W_n$ and a split epimorphism of left $\AW_{\chi}$-modules $\AW_{\chi, n\theta} \otimes W_n \twoheadrightarrow \AW_\chi$. Then, for every good $\AW_\chi$-module $\ms{M}$, we have $\mathbb{R}^i f_* (\ms{M}) = 0$ for $i \neq 0$.
\item Assume that for all $n \gg 0$ there exists a finite dimensional vector space $U_n$ and a split epimorphism of left $\AW_{\chi}$-modules $\AW_\chi \otimes U_n \twoheadrightarrow \AW_{\chi, n\theta}$. Then every good $\AW_\chi$-module is generated, locally on $Y_0$, by its global sections.
\end{enumerate}
\ethm

\begin{lem}\label{lem:shiftimpliesepi}
Let $\AW_\chi$ and $\U_\chi$ be as above and choose $\theta \in \X(G)$. 
\begin{enumerate}[(i)]
\item If $\chi \leftarrow \chi + \theta$ then there exists a finite dimensional vector space $W$ and a split epimorphism $\AW_{\chi,\theta} \otimes W \twoheadrightarrow \AW_\chi$.
\item If $\chi \rightarrow \chi + \theta$ then there exists a finite dimensional vector space $U$ and a split epimorphism $\AW_\chi \otimes U \twoheadrightarrow \AW_{\chi, \theta}$.
\end{enumerate}
\end{lem}

\begin{proof}
We begin with $(i)$. Equation (\ref{eq:someequality}) implies that we have a morphism 
$$
\U^{\theta}_\chi \ra \E^{\theta}_{\chi},
$$
which a direct calculation shows is a morphism of $(\U_\chi,\U_{\chi + \theta}) = (\E_\chi,\E_{\chi + \theta})$-bimodules (here we identify $\U_\chi$ with $\E_\chi$ via the isomorphism of Proposition \ref{lem:qhriso}). Thus $\chi \leftarrow \chi + \theta$ implies that
$$
\E^{\theta}_{\chi} \otimes \E^{-\theta}_{\chi + \theta} \twoheadrightarrow \E_{\chi}.
$$
Therefore there exists some $k$ and $\phi_i \in \E^{\theta}_{\chi}$, $\psi_i \in \E^{-\theta}_{\chi + \theta}$ for $i \in [1,k]$ such that
$$
(\msf{id}_{\mc{L}_{\chi}} \otimes \msf{ev}) \circ (\sum_{i = 1}^k  (\psi_i \otimes \msf{id}_{\C_{\theta}}) \circ \phi_i) = \msf{id}_{\mc{L}_\chi}.
$$
Let $W = \Span_{\C} \, \{ \psi_i \, : \, i \in [1,k] \}$ and define $\Psi \, : \L_{\chi + \theta} \otimes  \C_{\theta} \otimes W \rightarrow \L_\chi$ by
$$
\Psi(u \otimes \theta \otimes \psi) = (\msf{id}_{\mc{L}_\chi} \otimes \msf{ev})(\psi(u) \otimes \theta).
$$
The map $\tilde{\Psi} \, : \, \L_\chi \longrightarrow \L_{\chi + \theta} \otimes \C_{\theta} \otimes W$ defined by $v \mapsto \sum_{i = 1}^k \phi_i(v) \otimes \psi_i$ is a right inverse to $\Psi$. Hence $\Psi$ is a split epimorphism. Since $\Psi$ and $\tilde{\Psi}$ are $(G,\C^{\times})$-equivariant we can apply the functor $p_* \ms{H}om_{\widetilde{\W}_{\Xo}}( \L_\chi, - )^G$, which by \cite[Proposition 2.8 (ii)]{KR} is an equivalence, to the morphism $\L_{\chi + \theta} \otimes \C_{\theta} \otimes W \rightarrow \L_\chi$ to get the required (necessarily split, epic) morphism.\\
Part $(ii)$ is similar. Again using Proposition \ref{lem:qhriso}, $\chi \rightarrow \chi + \theta$ implies that
$$
\E^{-\theta}_{\chi + \theta} \otimes \E^{\theta}_{\chi} \twoheadrightarrow \E_{\chi + \theta}.
$$
Therefore there exists some $k$ and $\phi_i \in \E^{-\theta}_{\chi + \theta}$, $\psi_i \in \E^{\theta}_{\chi}$ for $i \in [1,k]$ such that
$$
(\msf{id}_{\mc{L}_{\chi + \theta}} \otimes \msf{ev}) \circ \Bigl( \sum_{i = 1}^k  (\psi_i \otimes \msf{id}_{\C_{-\theta}}) \circ \phi_i \Bigr) = \msf{id}_{\mc{L}_{\chi + \theta}}.
$$
Let $U = \Span_{\C} \, \{ \psi_i \, : \, i \in [1,k] \}$ and define $\Phi \, : \L_{\chi} \otimes U \longrightarrow \L_{\chi + \theta} \otimes \C_{\theta}$ by $\Phi(u \otimes \psi) = \psi(u)$. The map $\tilde{\Phi} \, : \, \L_{\chi + \theta} \otimes \C_{\theta} \longrightarrow \L_\chi \otimes U$ defined by
$$
v \mapsto (\msf{id}_{\L_{\chi}} \otimes \msf{id}_{U} \otimes \msf{ev}) \Bigl( \sum_{i = 1}^k \phi_i(v) \otimes \psi_i \Bigr)
$$
is a right inverse to $\Phi$. Hence $\Phi$ is a split epimorphism. Since $\Phi$ and $\tilde{\Phi}$ are $(G,\C^{\times})$-equivariant we can apply $p_* \ms{H}om_{\widetilde{\W}_{\Xo}}( \L_\chi, - )^G$ to the morphism $\L_{\chi} \otimes U \longrightarrow \L_{\chi + \theta} \otimes \C_{\theta}$ to get the required (necessarily split, epic) morphism.
\end{proof}

\begin{proof}[Proof of Theorem \ref{thm:equivalence}]
It follows from the equivalence \cite[Proposition 2.8 (iv)]{KR} that $\Gamma(Y_{\vartheta},\AW_\chi)^F = \E_\chi$. Therefore part (i) follows from Proposition \ref{lem:qhriso}. Lemma \ref{lem:shiftimpliesepi} and Theorem \ref{thm:KRConditions} show that $\chi \leftarrow \chi + n \theta$ for all $n \in \Z_{\ge 0}$ implies that $\mathbb{R}^i f_*(\ms{M}) = 0$ for all $i > 0$ and all $\ms{M} \in \Mod^{\mathrm{good}}_F(\AW_\chi)$. Similarly, $\chi \rightarrow \chi + n \theta$ for all $n \in \Z_{\ge 0}$ implies that every good $\AW_\chi$-module is generated, locally on $Y_0$, by its global sections. Let $o$ denote the image of the origin of $T^* V$ in $Y_0$. The $\C^{\times}$-action we have defined on $Y_0$ (via the $\C^{\times}$-action on $T^* V$) shrinks every point to $o$, in the sense that $\displaystyle \lim_{t \rightarrow \infty} T_t (y) = o$ for all $y \in Y_0$. In such a situation, \cite[Lemma 2.13]{KR} says that $\mathbb{R}^i f_*(\ms{M}) = 0$ for all $i > 0$ and all $\ms{M} \in \Mod^{\mathrm{good}}_F(\AW_\chi)$ implies that $\Hom_{\Mod^{\mathrm{good}}_F(\AW_\chi)}(\AW_\chi, - )$ is an exact functor. Similarly, \cite[Lemma 2.14]{KR} says that if every good $\AW_\chi$-module $\ms{M}$ is generated, locally on $Y_0$, by its global sections then every $\ms{M}$ is generated by its $F$-invariant global sections. That is,
$$
\AW_\chi \otimes_{\U_\chi} \Hom_{\Mod_F^{\mathrm{good}}(\AW_\chi)}(\AW_\chi, \ms{M} ) \twoheadrightarrow \ms{M}.
$$
With these facts one can follow the proof of \cite[Corollary 11.2.6]{HTT}, more or less word for word.
\end{proof}

\section{Hypertoric Varieties}\label{sec:geometry}

\subsection{} As we have seen in the previous section, when one has a reductive group $G$ acting on a vector space $V$, there exists a family of $W$-algebras on the Hamiltonian reduction of the cotangent bundle of $V$. The simplest such situation is where $G = \T$, a $d$-dimensional torus. In this case the corresponding Hamiltonian reduction is called a hypertoric variety. In this section we recall the definition of, and basic facts about, hypertoric varieties. The reader is advised to consult \cite{Proudfoot} for an excellent introduction to hypertoric varieties. Here we will follow the algebraic presentation given in \cite{HS}. Thus, in this section only, spaces will be algebraic varieties over $\C$ in the Zariski topology.

\subsection{Torus actions}\label{sec:defineaction} Fix $1 \le d < n \in \N$ and let $\T := (\C^{\times})^d$. We consider $\T$ acting algebraically on the $n$-dimensional vector space $V$. If we fix coordinates on $V$ such that the corresponding coordinate functions $x_1, \ds , x_n$ are eigenvectors for $\T$ then the action of $\T$ is encoded by a $d \times n$ integer valued matrix $A = [ a_1, \ds,  a_n] = (a_{ij})_{i \in [1,d]; j \in [1,n]}$ and is given by
$$
(\xi_1, \ds, \xi_d) \cdot x_i = \xi_1^{a_{1i}} \cdots \xi_d^{a_{di}} x_i, \quad \forall \, (\xi_1, \ds, \xi_d) \in \T.
$$
We fix the co-ordinate ring of $V$ to be $R := \C [ x_1, \ds, x_n]$. The algebra $R$ is graded by the action of $\T$, $\deg (x_i) = a_i$. We make the assumption that the $d \times d$ minors of $A$ are relatively prime. This ensures that the map $\Z^n \stackrel{A}{\longrightarrow} \Z^d$ is surjective and hence the stabilizer of a generic point is trivial.

\subsection{}\label{sec:4.3} Since $\Z^d$ is a free $\Z$-module, the above assumption implies that we can choose an $n \times (n-d)$ integer valued matrix $B = [b_1, \ds, b_n]^T$ so that the following sequence is exact:
\beq\label{eq:exact1}
0 \longrightarrow \Z^{n- d} \stackrel{B}{\longrightarrow} \Z^{n} \stackrel{A}{\longrightarrow} \Z^d = \X \longrightarrow 0,
\eeq
where, as before, $\X := \Hom_{gp}(\T,\C^{\times})$ is the character lattice of $\T$ and $\Z^n$ is identified with the character lattice of $(\C^{\times})^n \subset GL(\C^n)$. The dual $\Hom_{\Z}(\X,\Z)$ of $\X$, which parameterizes one-parameter subgroups of $\T$, will be denoted $\Y$. Applying the functor $\Hom( \, - \, , \C^{\times})$ to the sequence (\ref{eq:exact1}) gives a short exact sequence of abelian groups
\beq\label{eq:exact2}
1 \longrightarrow \T \stackrel{A^T}{\longrightarrow} (\C^{\times})^n \stackrel{B^T}{\longrightarrow} (\C^{\times})^{n-d} \longrightarrow 1.
\eeq
Let $\mf{t}$ denote the Lie algebra of $\T$ and $\mf{g}$ the Lie algebra of $(\C^{\times})^n$. Differentiating the sequence (\ref{eq:exact2}) produces the short exact sequence
\beq\label{eq:exact3}
0 \longrightarrow \mf{t} \stackrel{A^T}{\longrightarrow} \g \stackrel{B^T}{\longrightarrow} \mathrm{Lie} (\C^{\times})^{n-d} \longrightarrow 0,
\eeq
of abelian Lie algebras.

\subsection{Geometric Invariant Theory}\label{sec:GIT} The standard approach to defining ``sensible'' algebraic quotients of $V$ by $\T$ is to use geometric invariant theory. We recall here the basic construction that will be used. Let $\X_{\Q} := \X \otimes_{\Z} \Q$ be the space of fractional characters. We fix a stability parameter $\delta \in \X_{\Q}$. For $\underline{k} = (k_1, \ds, k_n) \in \N^n$, the monomial $x_1^{k_1} \cdots x_n^{k_n}$ will be written $x^{\underline{k}}$. Then $\lambda \cdot x^{\underline{k}} = \lambda^{A \cdot \underline{k}} x^{\underline{k}}$ and we define
$$
R^{\delta} := \textrm{Span}_{\C} \, ( x^{\underline{k}} \, | \, A \cdot \underline{k} = \delta ),
$$
to be the space of $\T$-semi-invariants of weight $\delta$. Note that $R^{\delta} = 0$ if $\delta \notin \X$. A point $p \in V$ is said to be $\delta$-\textit{semi-stable} if there exists an $n > 0$ such that $n \delta \in \X$ and $f \in R^{n \delta}$ with $f(p) \neq 0$. A point $p$ is called $\delta$-\textit{stable} if it is $\delta$-semi-stable and in addition its stabilizer under $\T$ is finite. The set of $\delta$-semi-stable points in $V$ will be denoted $V^{\mathrm{ss}}_\delta$. The parameter $\delta$ is said to be \textit{effective} if $R^{n \delta} \neq 0$ for some $n > 0$ (by the Nullstellensatz this is equivalent to $V^{\mathrm{ss}}_{\delta} \neq \emptyset$).

\begin{defn}\label{defn:GIT}
Let $\delta \in \X_\Q$ be an effective stability condition. The \textit{G.I.T quotient} of $V$ by $\T$ with respect to $\delta$ is the variety
$$
X(A,\delta) := \textrm{Proj} \, \bigoplus_{k \ge 0} R^{k \delta};
$$
it is projective over the affine quotient
$$
X(A,0) := \textrm{Spec} \, (R^{\T}).
$$
\end{defn}

If a point $p \in V$ is not $\delta$-semi-stable it is called $\delta$-\textit{unstable}. Using the one-parameter subgroups of $\T$ one can describe the set $V_\delta^{\mathrm{us}}$ of $\delta$-unstable points. We denote by $\langle - , - \rangle$ the natural pairing between $\Y$ and $\X$ (and by extension between $\mf{t}$ and $\mf{t}^*$). Let $V(f_1, \dots, f_k)$ denote the set of common zeros of the polynomials $f_1, \dots, f_k \in R$.

\begin{lem}\label{lem:unstable}
Let $\delta \in \X_{\Q}$ be an effective stability parameter. The $\delta$-unstable locus is
\beq\label{eq:unstable}
V_\delta^{\mathrm{us}} = \bigcup_{\stackrel{\lambda \in \Y}{\langle \lambda , \delta \rangle < 0}} V(x_i \, | \, \langle \lambda, a_i \rangle < 0 ).
\eeq
Moreover, there exists a finite set $\mc{F}(\delta) = \{ \lambda_1, \ds, \lambda_k \} \subset \Y$, $\langle \lambda_i , \delta \rangle < 0$ such that $$
\bigcup_{\stackrel{\lambda \in \Y}{\langle \lambda , \delta \rangle < 0}} V(x_i \, | \langle \lambda, a_i \rangle < 0 ) = \bigcup_{\lambda \in \mc{F}(\delta)}  V(x_i \, | \langle \lambda , a_i \rangle < 0 ).
$$
\end{lem}

\begin{proof}
Let $S := R[t]$ and extend the action of $\T$ from $R$ to $S$ by setting $g \cdot t = \delta(g)^{-1} t$, $\, \forall \, g \in \T$. Then $(S)^{\T} = \bigoplus_{n \ge 0} \, R^{n \delta} \cdot t^n$. Now
$$
\begin{array}{rcl}
u \in V^{\mathrm{us}}_{\delta} & \Longleftrightarrow & f(u) = 0 \quad \forall \, f \in R^{n \delta}, n > 0, \\
& \Longleftrightarrow & F(u,1) = 0 \quad \forall \, F \in (S^{\T})_+ = (S_+)^{\T}, \\
& \Longleftrightarrow & \overline{\T \cdot (u,1)} \cap V \times \{ 0 \} \neq \emptyset,
\end{array}
$$
where $(S^{\T})_+ = (S_+)^{\T}$ follows from the fact that $\T$ is reductive. Then \cite[Theorem 1.4]{KempfInvariants} says that there exists a one-parameter subgroup $\lambda \in \Y$ such that $\lim_{t \rightarrow 0} \lambda(t) \cdot (u,1) \in V \times \{ 0 \}$. Writing $u = u_1 + \ds + u_n$ such that $x_i(u) = u_i$, we have
$$
\lambda(t) \cdot (u,1) = \Bigl(\sum_{i = 1}^n t^{ - \langle \lambda, a_i \rangle } u_i , t^{\langle \lambda, \delta \rangle}\Bigr)
$$
which implies that $u_i = 0$ for all $i \in [1,n]$ such that $\langle \lambda, a_i \rangle > 0$ and $\langle \lambda, \delta \rangle > 0$. This shows that the left hand side of (\ref{eq:unstable}) is contained in the right hand side. Conversely, if $u$ is $\delta$-semi-stable then it is also $\phi$-semi-stable with respect to the action of the one dimensional torus $\lambda : \T \hookrightarrow \T$ on $V$, where $\phi$ is the character of $\T$ defined by $t \mapsto t^{\langle \lambda, \delta \rangle}$. \end{proof}

\subsection{} The variety $X(A,\delta)$ is a toric variety and, as shown in \cite[Corollary 2.7]{HS}, any semi-projective toric variety equipped with a fixed point is isomorphic to $X(A,\delta)$ for suitable $A$ and $\delta$. Fix $S \subset V$ and let $\delta_1, \delta_2 \in \X_\Q$ be two stability parameters such that $S^{\mathrm{ss}}_{\delta_1},S^{\mathrm{ss}}_{\delta_2} \neq 0$. Then $\delta_1$ and $\delta_2$ are said to be equivalent if $S^{\mathrm{ss}}_{\delta_1} = S^{\mathrm{ss}}_{\delta_2}$. The set of all $\rho$ equivalent to a fixed $\delta$ will be denoted $C(\delta)$. These equivalence classes form the relative interiors of the cones of a rational polyhedral fan $\Delta(\T, S)$, called the G.I.T. fan, in $\X_\Q$. The support of $\Delta(\T,S)$ is the set of all effective $\delta \in \X_\Q$ such that $S^{\mathrm{ss}}_\delta \neq 0$ and is denoted $|\Delta(\T, S)|$. We will mainly be concerned with $S = V$. The cones in $\Delta(\T,V)$ having the property that the stable locus is properly contained in the semi-stable locus are called the walls of $\Delta(\T,V)$. The G.I.T. fan is quite difficult to describe explicitly, see \cite{OdaPark}. However one has the following explicit description of the walls of $\Delta(\T,V)$.

\blem\label{lem:walls}
Let $\T$ act on $V$ via $A$ as in (\ref{sec:defineaction}). Then $|\Delta(\T,V)| = \sum_{i = 1}^n \Q_{\ge 0} \cdot a_i$ and the walls of the fan are $\sum_{i \in J} \Q_{\ge 0} \cdot a_i$, where $J \subset [1,n]$ is any subset such that $\dim_{\Q} (\textrm{Span}_{\Q} ( a_i \, | \, i \in J )) = d - 1$.
\elem

\begin{proof}
Let $0 \neq \delta \in \sum_{i = 1}^n \Z_{\ge 0} \cdot a_i$ and write $\delta = \sum_{i \in I} n_i a_i$ where $I \subset [1,n]$ and $n_i > 0$ for all $i \in I$. Then $0 \neq f = \prod_{i \in I} x_i^{n_i} \in R^\delta$ implies that $\delta$ is effective. Now let $\delta \in \X$ be any effective stability parameter and choose $0 \neq p \in V^{\mathrm{ss}}_\delta$. Write $p = p_1 + \ds + p_n$ so that $x_i(p)= p_i$ and let $I = \{ i \in [1,n] \, | \, p_i \neq 0 \}$. Then Lemma \ref{lem:unstable} shows that
$$
\begin{array}{rcl}
p \in V^{\mathrm{ss}}_{\delta} & \Longleftrightarrow & \left( \langle \lambda, \delta \rangle < 0 \, \Rightarrow \exists \, i \in I, \, \langle \lambda, a_i \rangle < 0 \right) \\
& \Longleftrightarrow & \{ \lambda \in \Y \, | \, \langle \lambda, \delta \rangle < 0 \} \cap \left( \sum_{i \in I} \Z_{\ge 0} \cdot a_i \right)^{\vee} = \emptyset \\
& \Longleftrightarrow & \left( \sum_{i \in I} \Z_{\ge 0} \cdot a_i \right)^{\vee} \subset  \{ \lambda \in \Y \, | \, \langle \lambda, \delta \rangle \ge 0 \} \\
& \Longleftrightarrow & \delta \in \sum_{i \in I} \Z_{\ge 0} \cdot a_i.
\end{array}
$$
Now choose $\delta \in \X$ to lie on a wall. By definition, there exists a $\delta$-semi-stable point $p$ such that $\dim \textrm{Stab}_{\, \T}(p) \ge 1$. Let $I$ be as above. Then $\dim \textrm{Stab}_{\, \T}(p) \ge 1$ implies that the subspace $\sum_{i \in I} \Q \cdot a_i$ must be a proper subspace of $\X_\Q$. The above reasoning shows that $\delta \in \sum_{i \in I} \Z_{\ge 0} \cdot a_i$ as required. \end{proof}
Lemma \ref{lem:walls} shows that, under our assumption on $A$, the maximal cones of $\Delta(\T,V)$ are all $d$-dimensional. We will refer to these maximal cones as the \textit{$d$-cones} of $\Delta(\T,V)$. The integer valued matrix $A$ is said to be \textit{unimodular}\footnote{In \cite{HS}, the authors define $A$ to be unimodular if every non-zero $d \times d$ minor of $A$ has the same absolute value. However they also, as do we, make the assumption that the $d \times d$ minors of $A$ are relatively prime. Thus, their definition agrees with ours.} if every $d \times d$ minor of $A$ takes values in $\{ - 1, 0 ,1\}$. Combining \cite[Corollary 2.7]{HS} and \cite[Corollary 2.9]{HS} gives:

\begin{thm}
The variety $X(A,\delta)$ is an orbifold if and only if $\delta$ belongs to the interior of a $d$-cone of $\Delta(\T,V)$. It is a smooth variety if and only if $\delta$ belongs to the interior of a $d$-cone of $\Delta(\T,V)$ and $A$ is unimodular.
\end{thm}

\subsection{Hypertoric varieties}\label{sec:hypertoric}

Define $A^{\pm} := [A,-A]$, a $d \times 2n$ matrix. It defines a grading on the ring $R := \C[T^* V] = \C[x_1, \ds, x_n, y_1, \ds, y_n]$. For $\delta$ in the interior of a $d$-cone of $\Delta(\T,T^* V)$, the corresponding toric variety $X(A^\pm,\delta)$ is called a Lawrence toric variety associated to $A$. It is a G.I.T. quotient of the symplectic vector space $T^* V$ with canonical symplectic form
$$
\omega = d x_1 \wedge d y_1 + \ds +  d x_n \wedge d y_n.
$$
The action of $\T$ is Hamiltonian and the moment map is given by
$$
\mu \, : \, T^* V \longrightarrow \mf{t}^*, \quad \mu(\mbf{x},\mbf{y}) = \Bigl(  \sum_{j = 1}^n  a_{ij} x_j y_j  \Bigr)_{i \in [1,d]}.
$$
Consider the ideal
$$
I := I( \mu^{-1}(0)) = \Bigl\langle \, \sum_{j = 1}^n a_{ij} x_j y_j \, : \, i \in [1,d] \Bigr\rangle \subset R,
$$
it is homogeneous and generated by $\T$-invariant polynomials.

\begin{defn}\label{defn:hypertoric}
The hypertoric variety associated to $A$ and $\delta$ is defined to be
$$
Y(A,\delta) := \mu^{-1}(0) /\!/_{\delta} \, \T = \textrm{Proj} \, \bigoplus_{k \ge 0 } \, ( R / I )^{k \delta};
$$
it is projective over the affine quotient
$$
Y(A,0) := \textrm{Spec} \, \left( ( R / I )^{\T} \right).
$$
\end{defn}

The basic properties of hypertoric varieties can be summarized as: 

\begin{prop}[\cite{HS}, Proposition 6.2]\label{prop:smoothres}
If $\delta$ is in the interior of a $d$-cone of $\Delta(\T,\mu^{-1}(0))$ then the hypertoric variety $Y(A,\delta)$ is an orbifold. It is smooth if and only if $\delta$ is in the interior of a $d$-cone of $\Delta(\T,\mu^{-1}(0))$ and $A$ is unimodular.
\end{prop}

\subsection{}

In this subsection we show that the assumptions of (\ref{sec:affinity}) are valid for hypertoric varieties. Let $f : Y(A,\delta) \ra Y(A,0)$ be the projective morphism from $Y(A,\delta)$ to $Y(A,0)$. Lemma \ref{lem:extdim2} below together with Proposition \ref{prop:smoothres} imply that $f$ is birational and hence a resolution of singularities when $Y(A,\delta)$ is smooth. The symplectic form $\omega$ on $T^* V$ induces a symplectic $2$-form on the smooth locus of $Y(A,\delta)$. In particular, when $Y(A,\delta)$ is smooth it is a symplectic manifold. Proposition \ref{prop:symplecticvariety} below shows that $Y(A,0)$ is a symplectic variety and the resolution $f$ is a symplectic resolution.\footnote{We refer the reader to \cite{FuSurvey} for the definition of symplectic variety and symplectic resolution.} This implies that $Y(A,0)$ is also normal. 

\begin{lem}\label{prop:compint}
The moment map is flat and $\mu^{-1}(0)$ is a reduced complete intersection in $T^* V$. If no row of the matrix $B$ is zero then $\mu^{-1}(0)$ is irreducible. 
\end{lem}

\begin{proof}
The lexicographic ordering on a monomial $\underline{x}^{\alpha}$, $\alpha \in \Z^{2n}$, is defined by saying that $\underline{x}^{\alpha} > \underline{x}^{\beta}$ if and only if the left most non-zero entry of $\alpha - \beta$ is greater than zero, see \cite[page 54]{IVA}. After permuting the variables $x_1, \ds, x_n$, we may assume that the first $d$ columns of $A$ are linearly independent. Applying an automorphism of $\T$ and then letting $\T$ act is the same as multiplying $A$ on the right by some unimodular $d \times d$-matrix. Using this fact we may assume that the left most $d \times d$-block of $A$ is the identity matrix. This allows us to rewrite the generators of $I$ as
$$
\Bigl\{ x_i y_i - \sum_{j = d+1}^n c_{i,j} x_j y_j \ \Bigm| \ i = 1, \ds , d \Bigr\},
$$
where $c_{i,j} \in \Z$. By \cite[Theorem 8, page 451]{IVA}, $\dim \mu^{-1}(0) = \dim V( \mathrm{in} (I))$, where $\mathrm{in} (I)$ denotes the initial ideal of $I$ with respect to the ordering $x_1 > x_2 > \ds > y_1 > y_2 \ds$. Now by \cite[Theorem 3 (Division algorithm), page 61]{IVA}, we have $\mathrm{in} (I) = \langle x_1 y_1, \ds, x_d y_d \rangle$. This is the zero set of a union of $2^d$ linear subspaces of $T^* V$ of dimension $2n - d$. Therefore $\dim \mu^{-1}(0) = 2n - d$ and it follows from \cite[Lemma 2.3]{Holland} that the moment map is flat. To prove that it is a complete intersection we must show that the generators of $I$ given above form a regular sequence in the polynomial ring $R$. Once again, it suffices to note that $x_1 y_1, \ds, x_d y_d$ is a regular sequence. Also, since the ideal $\mathrm{in} (I)$ is radical, the ideal $I$ is itself radical.  

Now note that, since the sequence (\ref{eq:exact1}) is exact, the matrix $B$ contains a row of zeros if and only if there exists an $i \in [1,..,d]$ such that $c_{i,j} = 0$ for all $j > d$. So, when $B$ contains no rows equal to zero we can write
$$
y_i = x_i^{-1} \sum_{j = d+1}^n c_{i,j} x_j y_j \mod \ I
$$
on the open set $\mu^{-1}(0) \backslash V(x_1 \cdots x_d)$. This shows that $\mu^{-1}(0)$ contains an open set isomorphic to $\mathbb{A}^{2n - d}$. We just need to show that this open set is dense. Since $\mu^{-1}(0)$ is a complete intersection,  it is pure dimensional. Therefore it suffices to show that the dimension of $\mu^{-1}(0) \cap V(x_1 \cdots x_d)$ is at most $2n - d - 1$. Consider $Y = \mu^{-1}(0) \cap V(x_1)$ and let $J = I(Y)$. We may assume without loss of generality that $c_{1,d+1} \neq 0$. Then $J$ is generated by $x_1$, $x_i y_i - \sum_{j = d + 2}^n c_{i,j} x_{j} y_j$ for $j = 2, \ds, d$ and $x_{d+1} y_{d+1} + \sum_{j = d+2}^n c_{1,j} x_{j} y_j$. Hence $\mathrm{in} (J) = \langle x_1, x_2 y_2, \ds , x_{d+1} y_{d+1} \rangle$, which defines a variety of dimension $2n - d - 1$ as required. 
\end{proof}

From now on we assume that no row of the matrix $B$ is zero.  

\begin{lem}\label{lem:extdim}
For any $A$, we have $\dim X(A^{\pm},0) = 2n - d$.
\end{lem}

\begin{proof}
Let $U = V \backslash V(x_1 \cdots x_n)$ and let $S_1 = \C[x_1^{\pm 1}, \ds , x_n^{\pm 1}]^{\T}$ denote the coordinate ring of the quotient $U / \T$. Let $F_1$ be the field of fractions of $S_1$. Let $S_2 = \C[X(A^{\pm},0)]$ and $F_2$ its field of fractions. We claim that $F_1 \subset F_2$. An element in $F_1$ is a fraction $f(x_1, \ds, x_n) / g(x_1, \ds, x_n)$, where $f$ and $g$ are homogeneous of the same weight with respect to $\T$. Then $f(\mbf{x})f(\mbf{y}), g(\mbf{x})f(\mbf{y}) \in S_2$ and $f(\mbf{x})f(\mbf{y}) / g(\mbf{x})f(\mbf{y}) = f(\mbf{x}) / g(\mbf{x})$ as required. Since $\dim \T^n / \T = n - d$, to prove the lemma it suffices to show that the field extension $F_1 \subset F_2$ has transcendental degree $n$. Consider the field $K = F_1 \langle x_1y_1, \ds, x_n y_n \rangle$. Then $F_1 \subset K \subset F_2$ and $K$ is a purely transcendental extension of $F_1$ of degree $n$. We claim that $K = F_2$. To show this it is sufficient to show that if $f \in S_2$ is a polynomial in the $x_i$'s and $y_j$'s then $f \in K$. We show more generally that if $f = f_1 / g$, where $f_1, g \in S_2$ and $g$ a monomial, then $f \in K$. We prove the claim by induction on the number of terms in $f$ (note that even though there is some choice in the exact form of each of the terms in $f$, the number of terms is unique). Let $u = \alpha x^{\mbf{i}} y^{\mbf{j}}$, $\mbf{i}, \mbf{j} \in \Z^n$, be some non-zero term of $f$. Then $(xy)^{-\mbf{j}} u \in F_1$ and $(xy)^{-\mbf{j}} f - (xy)^{-\mbf{j}} u \in K$ by induction. Since $(xy)^{-\mbf{j}} \in K$, this implies that $f \in K$.
\end{proof}

Note that, unlike $X(A^{\pm},0)$, the dimension of $X(A,0)$ can vary greatly depending on the specific entries of $A$.

\begin{lem}\label{lem:extdim2}
For any $A$, we have $\dim Y(A,0) = 2(n-d)$ and $Y(A,0)$ is Cohen-Macaulay.
\end{lem}

\begin{proof}
By Hochster's Theorem, \cite[Theorem 6.4.2]{CohMac}, the ring $\C[X(A^{\pm},0)]$ is Cohen-Macaulay. As noted in Lemma \ref{prop:compint}, the generators $u_1, \ds, u_d$ of $I$ form a regular sequence in $R$. Since $\T$ is reductive, $R = \C[X(A^{\pm},0)] \oplus E$ as a $\C[X(A^{\pm},0)]$-module. Therefore projection from $R$ to $\C[X(A^{\pm},0)]$ is a Reynolds operator in the sense of \cite[page 270]{CohMac}. Since $u_1, \ds, u_d$ are $\T$-invariant, \cite[Proposition 6.4.4]{CohMac} now says that they form a regular sequence in $X(A^{\pm},0)$. Therefore \cite[Theorem 2.1.3]{CohMac} says that $Y(A,0)$ is Cohen-Macaulay with $\dim Y(A,0) = \dim X(A^{\pm},0) - d$. The lemma follows from Lemma \ref{lem:extdim}.
\end{proof}

\begin{lem}
Let $\delta \in \X_{\Q}$. The graded ring $\bigoplus_{k \ge 0 } \, ( R / I )^{k \delta}$ is Cohen-Macaulay i.e. $Y(A,\delta)$ is arithmetically Cohen-Macaulay.
\end{lem}

\begin{proof}
Consider $S = R[t]$ with $\T$ acting on $t$ via $g \cdot t = \delta(g)^{-1} t$. Replacing $R$ with $S$ in the proof of Lemma \ref{lem:extdim2} gives a proof of the statement. 
\end{proof}

\begin{prop}\label{prop:symplecticvariety}
Let $A$ be unimodular and choose $\delta$ in the interior of a $d$-cone of $\Delta(\T,\mu^{-1}(0))$. Then $Y(A,0)$ is a symplectic variety and the morphism $f : Y(A,\delta) \ra Y(A,0)$ is a symplectic resolution. 
\end{prop}

\begin{proof}
The construction of $Y(A,\delta)$ and $Y(A,0)$ as Hamiltonian reductions means that they are Poisson varieties and $f$ preserves the Poisson structure. Therefore the smooth locus of $Y(A,0)$ is a symplectic manifold since $Y(A,\delta)$ is a symplectic manifold. In \cite[\S 2]{ProudWeb}, a stratification of $Y(A,0)$ into smooth locally closed subvarieties of even dimensions is constructed. This stratification shows that $Y(A,0)$ is smooth in co-dimension one. Therefore the fact (Lemma \ref{lem:extdim2}) that $Y(A,0)$ is Cohen-Macaulay together with Serre's normality criterion, \cite[Theorem 2.2.22]{CohMac} implies that $Y(A,0)$ is normal. Also, the fact that $Y(A,\delta)$ is a symplectic manifold implies that its canonical bundle is trivial. Therefore the Grauert-Riemenschneider vanishing theorem implies that $Y(A,0)$ has rational Gorenstein singularities. Then \cite[Theorem 6]{NamikawaExtensions} says that $Y(A,0)$ is a symplectic variety. 
\end{proof}

\subsection{G.I.T. chambers for hypertoric varieties} Define the subvariety $\mc{E}$ of $T^* V$ by $\mc{E} = \{ (x,y) \in T^* V \, | \, x_i \cdot y_i = 0 \, \forall \, i \in [1,n] \}$. We decompose $\mc{E}$ into its $n$-dimensional irreducible components
$$
\mc{E} = \bigcup_{I \subset [1,n]} \mc{E}_I; \qquad \mc{E}_I := \{ (x,y) \in T^* V \, | \, x_i = 0 \;\; \forall \, i \in I \textrm{ and } y_i = 0  \;\; \forall \, i \in [1,n] \backslash I \}.
$$
The subvariety $\mc{E}$ is preserved under the $\T$-action. Therefore we may consider the corresponding G.I.T. quotients. The G.I.T. quotient $\mc{E} /\!/_\delta \, \T$ is a closed subvariety of $Y(A,\delta)$, it is called the \textit{extended core} of $Y(A,\delta)$; see \cite{Proudfoot} for details.

\begin{lem}\label{lem:equalcones}
In $\X_\Q$ we have equalities of G.I.T. fans
$$
\Delta(\T,T^* V) = \Delta(\T,\mc{E}) = \Delta(\T,\mu^{-1}(0)).
$$
\end{lem}

\begin{proof}
For a fixed $I \subset [1,n]$ denote by $\pi_I \, : \, T^* V \twoheadrightarrow \mc{E}_I$ the projection that sends $x_i$ to zero if $i \in I$ and $y_j$ to zero if $j \in [1,n] \backslash I$. The restriction of $\pi_I$ to $\mu^{-1}(0)$ will be denoted $\widetilde{\pi}_I$. The statement of lemma follows from the claim:
\beq
\label{eq:412}
(T^* V)^{\mathrm{ss}}_{\delta} = \bigcup_{I \subset [1,n]} \pi_I^{-1}((\mc{E}_I)^{\mathrm{ss}}_{\delta}) \quad \textrm{ and } \quad (\mu^{-1}(0))^{\mathrm{ss}}_{\delta} = \bigcup_{I \subset [1,n]} (\widetilde{\pi}_I)^{-1}((\mc{E}_I)^{\mathrm{ss}}_{\delta})
\eeq
for each $\delta \in \X$. Let $p \in (\mc{E}_I)^{\mathrm{ss}}_{\delta}$. Then, without loss of generality, we may assume that there exists a monomial $f \in R^{N \delta}$, $N \ge 1$, such that $f(p) \neq 0$. Then $f(q) \neq 0$ for all $q \in \pi^{-1}_I(p)$. Hence $(T^* V)^{\mathrm{ss}}_{\delta} \supset \pi_I^{-1}((\mc{E}_I)^{\mathrm{ss}}_{\delta})$ and $(\mu^{-1}(0))^{\mathrm{ss}}_{\delta} \supset (\widetilde{\pi}_I)^{-1}((\mc{E}_I)^{\mathrm{ss}}_{\delta})$ for all $I \subset [1,n]$. Now choose $p \in (T^* V)^{\mathrm{ss}}_{\delta}$. Then there exist $m \in \N$ and $g \in R^{m \delta}$ such that $g(p) \neq 0$. We may assume without loss of generality that
$$
g = \prod_i x_i^{u_i} \prod_i y_i^{v_i}
$$
for some $u_i,v_i \ge 0$. By definition, $\sum_{i} (u_i - v_i) a_i = m \delta$. For each $i$, define $s_i$ and $t_i$ by
\begin{enumerate}
\item $u_i - v_i > 0 \Rightarrow s_i = u_i - v_i, t_i = 0$
\item $u_i - v_i < 0 \Rightarrow t_i = v_i - u_i, s_i = 0$
\item $u_i - v_i = 0 \Rightarrow s_i = t_i = 0$
\end{enumerate}
and $I = \{ i \in [1,n] \, | \, t_i \neq 0 \}$. Then $g(p) \neq 0$ implies that $\pi_I(p) \neq 0$. Define $\tilde{g} = \prod_i x_i^{s_i} \prod_i y_i^{t_i} \in R^{m \delta}$. Then $\tilde{g}(\pi_I(p)) \neq 0$ implies that $p \in \pi_I^{-1}((\mc{E}_I)^{\mathrm{ss}}_{\delta})$ and hence $(T^* V)^{\mathrm{ss}}_{\delta} = \bigcup_{I \subset [1,n]} \pi_I^{-1}((\mc{E}_I)^{\mathrm{ss}}_{\delta})$ as required. 
The second equality in (\ref{eq:412}) follows from the first one.
\end{proof}

\begin{cor}\label{cor:hypertoricwalls}
Let $\T$ act on $V$ via $A$ as in (\ref{sec:defineaction}). Then $|\Delta(\T,\mu^{-1}(0))| = \X_\Q$ and the walls of the fan $\Delta(\T,\mu^{-1}(0))$ are $\sum_{i \in J} \Q \cdot a_i$, where $J \subset [1,n]$ is any subset such that $\dim_{\Q} (\textrm{Span}_{\Q} ( a_i \, | \, i \in J )) = d - 1$.
\end{cor}

Assume now that $A$ is unimodular and choose $\delta \in \X$ to lie in the interior, denoted $C(\delta)$, of a $d$-cone of $\Delta(\T,\mu^{-1}(0))$. If $\zeta \in C(\delta) \cap \X$ then $Y(A,\delta) = Y(A,\zeta)$. Recall from (\ref{sec:affinity}) that $\zeta$ also defines a line bundle $L_\zeta$ on $Y(A,\delta)$. From the definition of $Y(A,\delta)$ as proj of a graded ring we see that $L_\zeta$ is an ample line bundle on $Y(A,\delta)$. Summarizing:

\begin{lem}\label{lem:amplebundle}
Let $A$ be unimodular and let $C(\delta)$ denote the interior of a $d$-cone of $\Delta(\T,\mu^{-1}(0))$. Then the line bundle $L_\zeta$ on $Y(A,\delta)$ is ample for all $\zeta \in C(\delta) \cap \X$.
\end{lem}

\section{Quantum Hamiltonian reduction}\label{sec:qhr}

\subsection{}\label{sec:QHRone} Recall that $\algD (V)$ denotes the ring of algebraic differential operators on the $n$-dimensional space $V$. Let $\T$ act on $V$ with weights described by the matrix $A$ (as in sections \ref{sec:defineaction} and \ref{sec:4.3}) and choose an element $\chi$ of the dual $\mf{t}^*$ of the Lie algebra $\mf{t}$ of $\T$. As explained in (\ref{subsection:defineqmm}), by differentiating the action of $\T$ we get a quantum moment map $\mu_D \, : \, \mf{t} \longrightarrow \algD (V)$, $t_i \mapsto \sum_{j = 1}^n a_{ij} x_j \p_j$. As in (\ref{sec:MQHR}), the quantum Hamiltonian reduction of $V$ with respect to $\chi$ is defined to be the non-commutative algebra
$$
\U_\chi := \left( \algD (V) \bigm/  \algD(V) (\mu_D - \chi)(\mf{t}) \right)^{\T}.
$$
We also have bimodules
$$
\U_{\chi}^\theta := \left( \algD (V) \bigm/ \algD (V) (\mu_D - (\chi + \theta))(\mf{t}) \right)^{-\theta}.
$$
We say that $\chi$ and $\chi + \theta$ are comparable if the multiplication map $\U_{\chi}^\theta \otimes \U_{\chi + \theta}^{-\theta} \longrightarrow \U_\chi$ is non-zero. By \cite[Theorem 7.3.1]{MVdB}, the ring $\U_{\chi}$ is a domain. Then \cite[Proposition 4.4.2]{MVdB} says that this implies that comparability is an equivalence relation. As in (\ref{sec:MQHR}), write $\chi \rightarrow \chi + \theta$ if $\U_{\chi + \theta}^{-\theta} \otimes \U_{\chi}^\theta \twoheadrightarrow \U_{\chi + \theta}$. As noted in \cite[Remark 4.4.3]{MVdB}, the relation $\rightarrow$ is transitive. Therefore it defines a pre-order on the set of elements in $\mf{t}^*$ comparable to $\chi$. We say that $\chi$ is \textit{maximal} if $\chi$ is maximal in this pre-ordering i.e. $\chi' \rightarrow \chi$ implies $\chi \rightarrow \chi'$.  

\subsection{The main results}

Write $\msf{pr} \, : \, \C \rightarrow \Q$ for the $\Q$-linear projection onto $\Q$ and denote by the same symbol the corresponding extension to $\mf{t}^*$:
$$
\msf{pr} \, : \, \mf{t}^* = \X \otimes_{\Z} \C \longrightarrow \X_{\Q}.
$$
We also write $\msf{pr}$ for the map $\C^n = \C \otimes_{\Z} \Z^n \rightarrow \Q^n$. Then $\msf{pr}(A \cdot v) = A \cdot \msf{pr}(v)$ for all $v \in \C^n$. The following proposition is the key to proving our main result. Its proof is given in subsection (\ref{sec:proofprop}).

\begin{prop}\label{prop:shiftingcone}
Let $C \subset \X_{\Q}$ be the interior of a $d$-cone in the fan $\Delta(\T,\mu^{-1}(0))$. Choose $\chi \in \mf{t}^*$ such that $\msf{pr} (\chi) \in C$. Then there exists a \textit{non-empty} $d$-dimensional integral cone $C(\chi) \subset C \cap \X \cup \{ 0 \}$ such that for all $\theta \in C(\chi)$, $\chi \leftarrow \chi + p \theta$ for all $p \in \Z_{\ge 0}$. 
\end{prop}

Recall from (\ref{sec:hypertoric}) and (\ref{subsection:defineqmm}) that for each $\chi \in \mf{t}^*$ and $\delta \in C$, where $C$ is the interior of a $d$-cone of $\Delta(\T,\mu^{-1}(0))$, we have defined the sheaf of algebras $\AW_\chi$ on the smooth symplectic manifold $Y(A,\delta)$.

\begin{thm}\label{thm:hypertoricequivalence}
Let $C \subset \X_{\Q}$ be the interior of a $d$-cone of $\Delta(\T,\mu^{-1}(0))$. Choose $\chi \in \mf{t}^*$ such that $\msf{pr} (\chi) \in C$ and choose $\delta \in C$. Let $\AW_\chi$ be the corresponding $W$-algebra on $Y(A,\delta)$.
\begin{enumerate}[(i)]
\item The functor $\Hom_{\Mod_F^{\mathrm{good}}(\AW_\chi)}(\AW_\chi , - )$ defines an equivalence of categories $$
\underline{\Mod}^{\mathrm{good}}_F(\AW_\chi) \, \stackrel{\sim}{\longrightarrow} \, \U_\chi \mmod
$$
with quasi-inverse $\AW_\chi \otimes_{\U_\chi} - $.
\item For any $0 \neq \theta \in C(\chi)$, there exists some $N > 0$ such that the functor $\Hom_{\Mod_F^{\mathrm{good}}(\AW_\chi)}(\AW_\chi , - )$ defines an equivalence of categories
$$
\Mod^{\mathrm{good}}_F(\AW_{\chi + N \theta}) \, \stackrel{\sim}{\longrightarrow} \, \U_{\chi + N \theta} \mmod
$$
with quasi-inverse $\AW_\chi \otimes_{\U_\chi} - $.
\end{enumerate}
\end{thm}

\begin{proof}
By Proposition \ref{prop:shiftingcone} we can choose $0 \neq \theta \in C(\chi)$ such that $\chi \leftarrow \chi + p \theta$ for all $p \in \Z_{\ge 0}$. Since $C(\chi) \backslash \{ 0 \} \subset C$, Lemma \ref{lem:amplebundle} says that $\theta$ defines an ample line bundle $L_{\theta}$ on $Y(A,\delta)$ and we have $\AW_{\chi,\theta}(0) / \AW_{\chi,\theta} (-1/2) \simeq L_{-\theta}$. Then part (i) of the theorem is a particular case of Theorem \ref{thm:equivalence} (ii). The proof of Proposition \ref{prop:shiftingcone} shows that we actually have 
$$
\chi \leftarrow \chi + \theta \leftarrow \chi + 2 \theta \leftarrow \cdots.
$$
Since the set of all covectors of the oriented matroid defined by $A$ is finite and each $\mc{Q}_\chi$ (which will be defined in Definition \ref{defn:lambdas2}) is a subset of this set, we see that there are only finitely many different $\mc{Q}_\chi$. Therefore we eventually get $\chi + N \theta \leftrightarrows \chi + (N + 1) \theta \leftrightarrows \cdots$ for some sufficiently large $N$. Then part (ii) of the theorem is a particular case of Theorem \ref{thm:equivalence} (iii) 
\end{proof}

\begin{cor}\label{cor:finiteglobaldim}
Let $Y(A,\delta),\AW_\chi,\U_\chi, \dots$ be as in Theorem \ref{thm:hypertoricequivalence} (with $\msf{pr} (\chi) \in C$). If the global dimension of $\U_{\chi}$ is finite then the functor $\Hom_{\Mod_F^{\mathrm{good}}(\AW_\chi)}(\AW_\chi , - )$ defines an equivalence of categories 
$$
\Mod^{\mathrm{good}}_F(\AW_\chi) \, \stackrel{\sim}{\longrightarrow} \, \U_\chi \mmod
$$
with quasi-inverse $\AW_\chi \otimes_{\U_\chi} - $.
\end{cor}

\begin{proof}
By Proposition \ref{prop:shiftingcone} we can choose $0 \neq \theta \in C(\chi)$ such that $\chi \leftarrow \chi + p \theta$ for all $p \in \Z_{\ge 0}$. However, \cite[Theorem 9.1.1]{MVdB} says that $\chi$ is maximal if and only if the global dimension of $\U_\chi$ is finite. Therefore $\chi \leftrightarrows \chi + \theta$ for all $\theta \in C(\chi)$. Then, as in the proof of Theorem \ref{thm:hypertoricequivalence}, Theorem \ref{thm:equivalence} implies the statement of the corollary.
\end{proof}

It seems natural to conjecture that, for any $\chi \in \mf{t}^*$, $\U_\chi$ has finite global dimension if and only if $\Mod^{\mathrm{good}}_F(\AW_\chi) \, \stackrel{\sim}{\longrightarrow} \, \U_\chi \mmod$, where $\AW_\chi$ is the corresponding $W$-algebra, defined on some $Y(A,\delta)$.

\subsection{The case $d = 1$} In specific cases it is possible to strengthen Proposition \ref{prop:shiftingcone}. One such case is when the torus $\T$ is one dimensional. Here the sets $\mc{Q}_\chi$ (which will be defined in Definition \ref{defn:lambdas2}) can be explicitly described, as was done in \cite{ToricVdB}. Since $A$ is assumed to be unimodular and $a_i \neq 0$ for all $i$ we see that $a_i = \pm 1$ for all $i$. After reordering we may assume that $a_1, \ds, a_k = 1$ and $a_{k+1}, \ds, a_n = -1$. For simplicity let us assume that $n > 1$. Then
$$
\begin{array}{lcl}
\mc{Q}_{\chi} = \{ 0 \} & \Leftrightarrow & \chi \in (\C \backslash \Z) \cup \{ k - n +1 , k - n + 2, \ds, n-k-2, n-k-1 \},\\
\mc{Q}_{\chi} = \{ 0, + \} & \Leftrightarrow & \chi \in \Z_{\ge n - k}, \\
\mc{Q}_{\chi} = \{ 0, - \} & \Leftrightarrow & \chi \in \Z_{\le k - n}.
\end{array}
$$
In this situation $\X_\Q = \Q$ and there are two $1$-cones with respect to the action of $\T$ on $\mu^{-1}(0)$, they are $\Q_{\ge 0}$ and $\Q_{\le 0}$. Applying Theorem \ref{thm:equivalence} gives

\begin{prop}
Let $\dim \T = 1$ and $n > 1$ and choose $\chi \in \mf{t}^*$. For $\delta \neq 0$, let $\AW_\chi$ denote the corresponding $W$-algebra on $Y(A,\delta)$. \begin{enumerate}[(i)]
\item When $\delta = 1$ we have an equivalence $\underline{\Mod}^{\mathrm{good}}_F(\AW_\chi) \, \stackrel{\sim}{\longrightarrow} \, \U_\chi \mmod$ if and only if $\chi \in (\C \backslash \Z) \cup \Z_{\ge 0}$ and $\underline{\Mod}^{\mathrm{good}}_F(\AW_\chi) = \Mod^{\mathrm{good}}_F(\AW_\chi)$ if and only if $\chi \in (\C \backslash \Z) \cup \Z_{\ge n-k}$.
\item When $\delta = -1$ we have an equivalence $\underline{\Mod}^{\mathrm{good}}_F(\AW_\chi) \, \stackrel{\sim}{\longrightarrow} \, \U_\chi \mmod$ if and only if $\chi \in (\C \backslash \Z) \cup \Z_{< 0}$ and $\underline{\Mod}^{\mathrm{good}}_F(\AW_\chi) = \Mod^{\mathrm{good}}_F(\AW_\chi)$ if and only if $\chi \in (\C \backslash \Z) \cup \Z_{\le k-n}$.
\end{enumerate}
\end{prop}

This result can be viewed as a variant of \cite[Theorem 6.1.3]{ToricVdB}, where sufficient conditions for the $D$-affinity of weighted projective spaces are stated.

\subsection{}\label{sec:proofprop}

The remainder of this section is devoted to the proof of Proposition \ref{prop:shiftingcone}. Since $\T$ can be considered as a subgroup of $\T^n$, $\mf{t}$ is a Lie subalgebra of $\mf{g} = \textrm{Lie} (\T^n)$ and we may regard elements of $\mf{t}$ as linear functionals on $\mf{g}^*$. Let $\rho : \mf{g}^* \twoheadrightarrow \mf{t}^*$ be the natural map. 

\bdefn\label{defn:lambdas}
Let $\lambda \in \Y$ and $\theta \in (\sum_{\langle \lambda, a_i \rangle = 0} \C \cdot a_i) \bigm/ (\sum_{\langle \lambda, a_i \rangle = 0} \Z \cdot a_i)$. We say that the pair $(\lambda, \theta)$ is attached to $\chi$ if there exists $\alpha \in \rho^{-1}(\chi)$ such that
$$
\sum_{\langle \lambda, a_i \rangle = 0} \alpha_i a_i \equiv \theta \, \mod \, \sum_{\langle \lambda, a_i \rangle = 0} \Z \cdot a_i
$$
and
$$
\begin{array}{rcl}\label{array:lambda}
\langle \lambda , a_i \rangle > 0 & \Rightarrow & \alpha_i \in \Z, \alpha_i \ge 0 \\
\langle \lambda , a_i \rangle < 0 & \Rightarrow & \alpha_i \in \Z, \alpha_i < 0 \\
\langle \lambda , a_i \rangle = 0 & \Rightarrow & \alpha_i \in \C \backslash \Z.
\end{array}
$$
\edefn

\begin{remark} 
The above definition is based on \cite[Definition 7.2.1]{MVdB}. There it is stipulated that $\lambda \in \mf{t} \cap \Q^n$ but we only care about whether $\langle \lambda , a_i \rangle$ is $>0, <0$ or $ = 0$ therefore we can assume $\lambda \in \Y$. Also our sign convention in Definition \ref{defn:lambdas} is opposite to that given in \cite[Definition 7.2.1]{MVdB} so that it agrees with the conventions of Section \ref{sec:geometry}.
\end{remark}

\subsection{} Let us define an equivalence relation on the set of pairs $(\lambda, \theta)$ by saying that $(\lambda_1, \theta_1)$ is equivalent to $(\lambda_2, \theta_2)$ if $\{ i \, | \, \langle \lambda_1 , a_i \rangle > 0 \} = \{ i \, | \, \langle \lambda_2 , a_i \rangle > 0 \}$, $\{ i \, | \, \langle \lambda_1 , a_i \rangle < 0 \} = \{ i \, | \, \langle \lambda_2 , a_i \rangle < 0 \}$ and $\theta_2 \equiv \theta_1 \, \mod \, \sum_{\langle \lambda_1 , a_i \rangle = 0} \Z \cdot a_i$. Denote by $\mc{P}_{\chi}$ the set of equivalence classes of pairs $(\lambda,\theta)$ that are attached to $\chi$. The set of all possible $\lambda$ up to equivalence consist of the (finitely many) covectors of the oriented matroid defined by $A$. It will be convenient to parameterize each $\lambda$ (again up to equivalence) as an element in $\{ + , 0 , - \}^n$, $\lambda \leftrightarrow (e_i)_{i \in [1,n]}$ with $e_i = +$ if $\langle \lambda, a_i \rangle > 0 $ and so forth. Note, however, that not every element of $\{ + , 0 , - \}^n$ can be realized as some $\lambda$.

\begin{prop}[Proposition 7.7.1, \cite{MVdB}]\label{prop:primitive}
Choose $\chi, \chi' \in \mf{t}^*$, then the set $\mc{P}_\chi$ parameterizes the primitive ideals in $\U_\chi$ and $\chi \rightarrow \chi'$ if and only if $\mc{P}_{\chi'} \subseteq \mc{P}_\chi$.
\end{prop}

Since we are interested in sheaves of $W$-algebras on smooth hypertoric varieties we may assume that $A$ is unimodular. This allows us to remove $\theta$ from the description of $\mc{P}_\chi$.

\begin{lem}\label{lem:removetheta}
Assume that $A$ is unimodular and let $(\lambda, \theta)$, respectively $(\lambda, \vartheta)$, be attached to $\chi$ via $\alpha \in \rho^{-1}(\chi)$, respectively via $\beta \in \rho^{-1}(\chi)$. Then $(\lambda, \theta)$ is equivalent to $(\lambda, \vartheta)$.
\end{lem}

\begin{proof}
By definition, $\theta$ is the equivalence class of $\sum_{\langle \lambda , a_i \rangle = 0} \alpha_{i} a_i$ in $(\sum_{\langle \lambda, a_i \rangle = 0} \C \cdot a_i) \bigm/ (\sum_{\langle \lambda, a_i \rangle = 0} \Z \cdot a_i)$, and similarly for $\vartheta$. Therefore we must show that $\sum_{i = 1}^n \alpha_i a_i = \sum_{i = 1}^n \beta_i a_i$ implies that
\beq\label{eq:equivequality}
\sum_{\langle \lambda , a_i \rangle = 0}^n \alpha_i a_i \equiv \sum_{\langle \lambda , a_i \rangle = 0}^n \beta_i a_i \, \mod \, \sum_{\langle \lambda, a_i \rangle = 0} \Z \cdot a_i.
\eeq
Choose $\{ a_{i_1}, \ds, a_{i_k} \} \subset \{ a_i \, | \, \langle \lambda , a_i \rangle = 0 \}$ to be a basis of the space spanned by the set $\{ a_i \, | \, \langle \lambda , a_i \rangle = 0 \}$. We can extend this to a basis $a_{i_1} ,\ds, a_{i_k} , a_{i_{k+1}}, \ds, a_{d}$ of $\mf{t}^*$. Since $A$ is unimodular the determinant of this basis is $\pm 1$. Hence $\{ a_{i_1}, \ds, a_{i_k} \}$ span a direct summand of the lattice $\X$. This implies that
\beq\label{eq:liesinZ}
\Bigl( \sum_{\langle \lambda , a_i \rangle = 0} \C \cdot a_i \Bigr) \cap \X = \sum_{\langle \lambda , a_i \rangle = 0} \Z \cdot a_i,
\eeq
which in turn implies (\ref{eq:equivequality}).
\end{proof}

It is shown in \cite[Example 7.2.7]{MVdB} that $\theta$ is not defined up to equivalence by $\chi$ and $\lambda$ if $A$ is not unimodular.

\subsection{}\label{sec:newdefQ}
Based on Lemma \ref{lem:removetheta} we make the following definition.
\bdefn\label{defn:lambdas2}
Let $\lambda \in \Y$ and $\chi \in \mf{t}^*$. We say that $\lambda$ is attached to $\chi$ if there exists $\alpha \in \rho^{-1}(\chi)$ such that
$$
\begin{array}{rcl}
\langle \lambda , a_i \rangle > 0 & \Rightarrow & \alpha_i \in \Z, \alpha_i \ge 0 \\
\langle \lambda , a_i \rangle < 0 & \Rightarrow & \alpha_i \in \Z, \alpha_i < 0 \\
\langle \lambda , a_i \rangle = 0 & \Rightarrow & \alpha_i \in \C \backslash \Z.
\end{array}
$$
If $\lambda_1,\lambda_2 \in \Y$ are attached to $\chi$ then we say that $\lambda_1$ is equivalent to $\lambda_2$ if $\{ i \, | \, \langle \lambda_1 , a_i \rangle > 0 \} = \{ i \, | \, \langle \lambda_2 , a_i \rangle > 0 \}$ and $\{ i \, | \, \langle \lambda_1 , a_i \rangle < 0 \} = \{ i \, | \, \langle \lambda_2 , a_i \rangle < 0 \}$. Let $\mc{Q}_\chi$ denote the set of equivalence classes of elements in $\Y$ that are attached to $\chi$. 
\edefn

\begin{lem}\label{lem:newprimitive}
Assume that $A$ is unimodular. Then $\chi \rightarrow \chi'$ if and only if $\mc{Q}_{\chi'} \subseteq \mc{Q}_{\chi}$ and $\chi - \chi' \in \X$.
\end{lem}

\begin{proof}
If $\chi \rightarrow \chi'$ then clearly $\chi - \chi' \in \X$ and Proposition \ref{prop:primitive} implies that $\mc{P}_{\chi'} \subseteq \mc{P}_{\chi}$. This implies that $\mc{Q}_{\chi'} \subseteq \mc{Q}_{\chi}$. 

Now assume that $\mc{Q}_{\chi'} \subseteq \mc{Q}_{\chi}$ and $\chi - \chi' \in \X$. Let $\lambda \in \mc{Q}_{\chi'}$ and choose $\alpha \in \rho^{-1}(\chi')$, respectively $\beta \in \rho^{-1}(\chi)$, satisfying the conditions of Definition \ref{defn:lambdas2} for $\lambda$ with respect to $\chi'$, respectively $\chi$. Write $\alpha = \alpha^{(1)} + \alpha^{(2)}$, where $(\alpha^{(1)})_i = \alpha_i$ if $\langle \lambda, a_i \rangle \neq 0$ and $(\alpha^{(1)})_i = 0$ if $\langle \lambda, a_i \rangle = 0$. Decompose $\beta = \beta^{(1)} + \beta^{(2)}$ in a similar fashion. Then 
$$
(\chi - \chi') - \rho(\beta^{(1)} - \alpha^{(1)}) \in \Bigl( \sum_{\langle \lambda , a_i \rangle = 0} \C \cdot a_i \Bigr) \cap \X,
$$
which, by (\ref{eq:liesinZ}), equals $\sum_{\langle \lambda , a_i \rangle = 0} \Z \cdot a_i$. Therefore we can choose $u \in \Z^n$ such that $u_i = 0$ for all $i$ such that $\langle \lambda, a_i \rangle = 0$ and $\rho(u) = (\chi - \chi') - \rho(\beta^{(1)} - \alpha^{(1)})$. Define $\delta^{(2)} = \alpha^{(2)} + u$ and $\delta = \beta^{(1)} + \delta^{(2)}$ so that $\rho(\delta) = \chi$. We have 
$$
\bar{\delta}^{(2)} = \bar{\alpha}^{(2)} \in \Bigl( \sum_{\langle \lambda, a_i \rangle = 0} \C \cdot a_i \Bigr) \bigm/ \Bigl( \sum_{\langle \lambda, a_i \rangle = 0} \Z \cdot a_i \Bigr)
$$
and $(\lambda,\bar{\alpha}^{(2)})$, respectively $(\lambda,\bar{\delta}^{(2)})$, is attached to $\chi'$, respectively to $\chi$, in the sense of Definition \ref{defn:lambdas}. Therefore $(\lambda,\bar{\delta}^{(2)}) = (\lambda,\bar{\alpha}^{(2)}) \in \mc{P}_{\chi}$ implies that $\mc{P}_{\chi'} \subseteq \mc{P}_{\chi}$. Hence Proposition \ref{prop:primitive} implies that $\chi \rightarrow \chi'$. 
\end{proof}

\begin{proof}[Proof of Proposition \ref{prop:shiftingcone}]
As was stated in subsection \ref{sec:GIT}, the cone $\overline{C}$ is a rational cone. Therefore we can choose $\mu_1 , \ds, \mu_k$ in $\Y$ such that
$$
\overline{C} = \{ \chi \in \X_{\R} \, | \, \langle \mu_i , \chi \rangle \ge 0, \, \forall i \in [1,k] \} \supset \{ \chi \in \X_{\R} \, | \, \langle \mu_i , \chi \rangle > 0, \, \forall i \in [1,k] \} = C.
$$
We will construct $C(\chi)$ in three stages.

\subsection*{Claim $1$} There exists an integer $N_0 \gg 0$ such that $p N_0 \cdot \msf{pr}(\chi) \in \X \cap C$ and $\chi + p N_0 \cdot \msf{pr}(\chi) \rightarrow \chi$ for all $p \in \N$.\\

For each $\lambda \in \mc{Q}_\chi$ fix an element $\beta^\lambda \in \rho^{-1}(\chi)$ such that $\beta^\lambda$ satisfies the properties listed in Definition \ref{defn:lambdas2} with respect to $\lambda$. Then $\msf{pr}(\chi) = \sum_{i = 1}^n \msf{pr}(\beta^\lambda_i) a_i$ and we choose $N_0$ such that $N_0 \cdot \msf{pr}(\beta_i^\lambda) \in \Z$ for all $\lambda \in \mc{Q}_{\chi}$ and all $i$. The element $(\beta^\lambda_i + p N_0 \beta^\lambda_i)_{i \in [1,n]}$ in $\mf{g}^*$ satisfies the properties of Definition \ref{defn:lambdas2} with respect to $\lambda$ hence $\mc{Q}_{\chi} \subseteq \mc{Q}_{\chi + p N_0 \cdot \msf{pr}(\chi)}$. Since $p N_0 \cdot \msf{pr}(\chi) \in \X$, Lemma \ref{lem:newprimitive} says that $\chi + p N_0 \cdot \msf{pr}(\chi) \rightarrow \chi$ for all $p \in \N$. Note also that $\langle \mu_i, \msf{pr}(\chi + p N_0 \cdot \msf{pr}(\chi)) \rangle = (1 + p N_0) \langle \mu_i, \msf{pr}(\chi) \rangle > 0$ for all $i$ shows that $\msf{pr}(\chi + p N_0 \cdot \msf{pr}(\chi)) \in C$.

\subsection*{Claim $2$} Fix $\delta = \sum_{i=1}^n a_i \in \X$. There exists an integer $N_1 \gg 0$ such that $N_1 \cdot \msf{pr}(\chi) + \delta \in \X \cap C$ and $\chi + p(N_1 \cdot \msf{pr}(\chi) + \delta) \rightarrow \chi$ for all $p \in \N$. Moreover, for all $\lambda \in \mc{P}_{\chi}$, there exists $\beta^{\lambda}$ as before except that $\beta_i^\lambda \neq 0$ for all $i$.\\

Choose $N_1 = p N_0$ such that
\beq\label{eq:bigN}
(N_1 / d) \cdot \langle \mu_i , \msf{pr}(\chi) \rangle > | \langle \mu_i , a_j \rangle|
\eeq
for all $i \in [1,k]$ and $j \in [1,n]$. Let $\beta^\lambda \in \rho^{-1}(\chi + N_1 \cdot \msf{pr}(\chi))$ satisfy Definition \ref{defn:lambdas2} with respect to $\lambda$ for $\lambda \in \mc{Q}_{\chi}$. By choosing a larger $p$ if necessary we may assume that $\beta^\lambda_i \in \Z \backslash \{ 0 \}$ implies that $| \beta^\lambda_i | > 1$. Then $\beta^\lambda_i + 1 < 0 $ if $\beta^\lambda_i \in \Z_{< 0}$ and $\beta^\lambda_i + 1 > 0 $ if $\beta^\lambda_i \in \Z_{\ge 0}$. Moreover $(\beta^\lambda_i + 1)_{i \in [1,n]}$ satisfies (\ref{array:lambda}) with respect to $\lambda$, $\sum_{i = 1}^n (\beta^{\lambda}_i + 1)a_i = \chi + (N_1 \cdot \msf{pr}(\chi) + \delta)$ and hence $\chi + (N_1 \cdot \msf{pr}(\chi) + \delta) \rightarrow \chi$. The same holds for all $\chi + p(N_1 \cdot \msf{pr}(\chi) + \delta)$. Finally (\ref{eq:bigN}) implies that $\msf{pr}(\chi + q(N_1 \cdot \msf{pr}(\chi) + \delta)) \in C$ for all $q \in \Z_{\ge 0}$.
\subsection*{Proof of the proposition}

Note that (\ref{eq:bigN}) implies that $p(N_1 \cdot \msf{pr}(\chi) + \delta) \in C$ for all $p$ as well. Let
$$
I_{\chi} = \{ \epsilon \in (-1 , 1)^n \subset \Q^n \, | \, - \langle \mu_i, (N_1 \cdot \msf{pr}(\chi) + \delta) \rangle < \langle \mu_i , \epsilon \rangle, \, \forall i \}.
$$
Since $C$ is $d$-dimensional, there exists some $0 < c < 1$ such that $[-c, c]^n \subset I_{\chi}$. Let $\{ v_j | \, j \in [1, 2^n] \} \subset [-c, c]^n$ be the vertices of the box. Choose $p \in \N$ such that $p \cdot v_j \in \Z^n$ for all $j$. The same argument as in Claims $1$ and $2$ shows that
$$
\chi \leftarrow \chi + q (p (N_1 \cdot \msf{pr}(\chi) + \delta + A \cdot v_j)), \quad \forall q \in \Z_{\ge 0}.
$$
We set $u_j = p (N_1 \cdot \msf{pr}(\chi) + \delta + A \cdot v_j)$. One can check as above that $\chi \leftarrow \chi + \sum_{i = 1}^{2^n} k_i \cdot u_i$ for all $k_i \in \Z_{\ge 0}$.  \end{proof}

\begin{remark}\label{remark:list}
We conclude with a couple of remarks regarding Proposition \ref{prop:shiftingcone}.
\begin{enumerate}[(i)]
\item Note that in the proof of Proposition \ref{prop:shiftingcone} we only used the fact that $C$ is the interior of some $d$-dimensional rational cone.
\item In general, the proposition is false when $\msf{pr}(\chi) \in C$ is replaced by $\msf{pr}(\chi) \in \overline{C}$. \item It would be very interesting to directly relate the sets $\mc{Q}_\chi$ to the G.I.T. fan.
\end{enumerate}
\end{remark}

\section{The rational Cherednik algebra associated to cyclic groups}

\subsection{}
As explained in the introduction, the original motivation for this article was to reproduce the results of \cite{KR} for the rational Cherednik algebra $H_{\mbf{h}}(\Z_m)$ associated to the cyclic group $\Z_m$. These rational Cherednik algebras are parameterized\ff{In this paper the parameters $(h_i)$ and $(\chi_i)$ are used. However the paper \cite{Kuwabarapreprint} uses the parameters $(\kappa_i)$ and $(c_i)$. The different parameterizations are related by $h_i \leftrightarrow \kappa_i$ and $c_i \leftrightarrow \chi_i - \chi_{i+1}$.} by an $m$-tuple $\mbf{h} = (h_i)_{i \in [0,m-1]} \in \C^m$, where the indices are taken modulo $m$. We fix a one dimensional space $\h = \C \cdot y$ and $\h^* = \C \cdot x$ such that $\langle x , y\rangle = 1$. The cyclic group $\Z_m = \langle \varepsilon \rangle$ acts on $\h$ and $\h^*$ via $\varepsilon \cdot y = \zeta^{-1} y$ and $\varepsilon \cdot x = \zeta x$, where $\zeta$ is a fix primitive $m^{th}$ root of unity. The idempotents in $\C \, \Z_m$ corresponding to the simple $\Z_m$-modules are $e_i = \frac{1}{m} \sum_{j = 0}^{m-1} \zeta^{-ij} \varepsilon^j$, $i \in [0,m-1]$, so that $\varepsilon \cdot e_i = \zeta^i e_i$. Then $e_{i+1} \cdot x = x \cdot e_i$ and $e_{i-1} \cdot y = y \cdot e_i$. If we fix $\alpha_{\varepsilon^i} = \sqrt{2} \cdot x$ and $\alpha_{\varepsilon^i}^\vee = \frac{-1}{\sqrt{2}} \cdot y$ then the commutation relations defining $H_{\mbf{h}}(\Z_m)$, as stated in \cite{RouquierQSchur}, become:
$$
\begin{array}{rcl}
\varepsilon \cdot x & = & \zeta x \cdot \varepsilon \\
\varepsilon \cdot y & = & \zeta^{-1} y \cdot \varepsilon \\
\lbrack y , x \rbrack & = & 1 + m \sum_{i = 0}^{m-1} (h_{i+1} - h_i) e_i,
\end{array}
$$
where indices are take modulo $m$.

\subsection{}\label{sec:categoryO} Category $\mc{O} \subset H_{\mbf{h}} \mmod$ is defined to be the subcategory of all finitely generated $H_{\mbf{h}}$-modules such that $y \in \C[\h^*]$ acts locally nilpotently. It is a highest weight category. To each simple $\Z_m$-module $\C \cdot e_i$, one can associate a \textit{standard module} in category $\mc{O}$ defined by
$$
\Delta(e_i) := H_{\mbf{h}} \otimes_{\C[\h^*] \rtimes \Z_m} \C \cdot e_i,
$$
where $y \in \C[\h^*]$ acts as zero on $e_i$. Each $\Delta(e_i)$ has a simple head $L(e_i)$ and $L(e_i) \not\simeq L(e_j)$ for $i \neq j$. The set of simple modules $\{ L(e_i) \}_{i \in [0,m-1]}$ are, up to isomorphism, all simple modules in $\mc{O}$. Fix $i \in [0,m-1]$ and let $c_i$ be the smallest element in $\Z_{\ge 1} \cup \{ \infty \}$ such that $c_i + m h_{i + c_i} - m h_i = 0$. The identity
$$
[y,x^j] = x^{j-1} \Bigl( j + m \sum_{i = 0}^{m-1} (h_{i + j} - h_i) e_i \Bigr), \, \forall j \ge 0,
$$
shows that $L(e_i) = (\C[x] / (x^{c_i})) \otimes e_{i}$. Fix $e := e_0$, the trivial idempotent. The algebra $e H_{\mbf{h}} e$ is called the spherical subalgebra of $H_{\mbf{h}}$. Multiplication by $e$ defines a functor $e \, : \, H_{\mbf{h}} \mmod \longrightarrow e H_{\mbf{h}} e \mmod$ with left adjoint $H_{\mbf{h}} e \otimes_{e H_{\mbf{h}} e} - $. Let $\mc{C} \subset \C^{m}$ be the union of the finitely many hyperplanes defined by the equations $j + m h_{i+j} - m h_{i} = 0$, where $i \in [1,m-1]$ and $j \in [0,m - i]$.

\begin{lem}\label{lem:sphericalequivalence}
The functor $e \, : \, H_{\mbf{h}} \mmod \longrightarrow e H_{\mbf{h}} e \mmod$ is an equivalence if and only if $\mbf{h} \notin \mc{C}$. This implies that $e H_{\mbf{h}} e$ has finite global dimension when $\mbf{h} \notin \mc{C}$.
\end{lem}

\begin{proof}
The functor $e$ will be an equivalence if and only if $H_{\mbf{h}} e H_{\mbf{h}} = H_{\mbf{h}}$. By Ginzburg's Generalized Duflo Theorem \cite[Theorem 2.3]{Primitive}, $H_{\mbf{h}} e H_{\mbf{h}} \neq H_{\mbf{h}}$ implies that there is some simple module in category $\mc{O}$ that is annihilated by $e$. This happens if and only if $\mbf{h} \in \mc{C}$. The second statement follows from the fact that $H_{\mbf{h}}$ has finite global dimension.
\end{proof}

\subsection{The minimal resolution of $\C^2 / \Z_m$}

In order to relate the spherical subalgebra of $H_{\mbf{h}}$ to a $W$-algebra on the resolution of the corresponding Kleinian singularity $\C^2 / \Z_m$ we must describe $e H_{\mbf{h}} e$ as a quantum Hamiltonian reduction. Such an isomorphism is well known and is a particular case of a more general construction by Holland \cite{Holland}. First we describe the minimal resolution of $\C^2 / \Z_m$ as a hypertoric variety. Let $Q$ be the cyclic quiver with vertices $V = \{ v_0, \ds, v_{m-1} \}$ and arrows $u_i : v_{i-1} \rightarrow v_i$ for $i \in [1,m]$ (where $v_m$ is identified with $v_0$). Let $\nu$ be the dimension vector with $1$ at each vertex. Then the space of representations for $Q$ with dimension vector $\nu$ is the affine space
$$
\Rep (Q,\nu) = \{ (u_i)_{i \in [1,m]} \, | \, u_i \in \C \} \simeq \C^m
$$
and we write $\C [ \Rep (Q,\nu)] = \C[x_1, \ds , x_m]$. There is an action of $\T^m = \{ (\lambda_i)_{i \in [1,m]} \, | \, \lambda_i \in \C^{\times} \}$ on $\Rep (Q,\nu)$ given by
$$
\lambda \cdot u_i = \lambda_i \lambda^{-1}_{i-1} u_i \quad \textrm{ and hence } \quad \lambda \cdot x_i = \lambda^{-1}_i \lambda_{i-1} x_i.
$$
The one dimensional torus $\T$ embedded diagonally in $\T^m$ acts trivially on $\Rep (Q,\nu)$. Therefore $\T^{m-1} := \T^m / \T$ acts on $\Rep (Q,\nu)$. The lattice of characters $\X(\T^{m-1})$ is the sublattice of $\X(\T^m) = \bigoplus_{i = 1}^{m} \Z \cdot v_i$ consisting of points $\phi = \sum_{i = 1}^m \phi_i v_i$ such that $\sum_{i = 1}^m \phi_i = 0$. We fix the basis $\{ w_i = v_i - v_{i + 1} \, | \, i \in [0,m-2] \}$ of $\X(\T^{m-1})$ so that $\phi = (\phi_i)_{i \in [1,n]} = \sum_{i = 1}^{m-1} \chi_i w_i$, where $\chi_i = \sum_{j = 1}^i \phi_i$. Then the $(m-1) \times m$ matrix encoding the action of $\T^{m-1}$ is given by
$$
A = (a_1, \ds , a_m) = \left(
\begin{array}{ccccc}
1 & 0 & \ldots & 0 & -1 \\
0 & 1 & & \vdots & -1 \\
\vdots & & \ddots & 0 & \vdots \\
0 & \ldots & 0 & 1 & -1
\end{array}
\right)
$$

The G.I.T walls in $\mf{t}^*$, where $\mf{t} = \textrm{Lie} ( \T^{m-1})$, are given by the hyperplanes $H_i = ( \chi_i = 0)$, $i \in [1,m-1]$, and $H_{ij} = ( \chi_i = \chi_j )$, $i \neq j \in [1,m-1]$. Hence the $m$-cones are the connected components of the complement to this union of hyperplanes. As was shown originally in terms of hyperk\"ahler manifolds by Kronheimer \cite{Kronheimer} and then by Cassens and Slodowy \cite{CassensSlodowy} in the algebraic setting:

\begin{prop}\label{prop:Kleinianresolution}
Let $\delta$ belong to the interior of an $m$-cone. Then the hypertoric variety $Y(A,\delta)$ is isomorphic to the minimal resolution $\widetilde{\C^2 / \Z_m}$ of the Kleinian singularity $\C^2 / \Z_m$.
\end{prop}

As is well-known, the hypertoric variety $Y(A,\delta)$ is a toric variety. It is shown in \cite[Theorem 10.1]{HS} that a hypertoric variety is toric if and only if it is a product of varieties of the form $\widetilde{\C^2 / \Z_m}$ for various $m$. Let us now consider the corresponding quantum Hamiltonian reduction
$$
\U_{\chi} = (\algD (\Rp) \bigm/ \algD (\Rp) (\mu_D - \chi)(\mf{t}))^{\T^{m-1}}.
$$
The quantum moment map in this case is given by
$$
\mu_D \, : \, \mf{t} \longrightarrow \algD (\Rp), \quad t_i \mapsto x_i \p_i - x_m \p_m,  \quad \forall i \in [1,m-1].
$$
Since $\algD (\Rp)^{\T^{m-1}} = \langle \p_1 \cdots \p_{m} , x_1 \cdots x_{m}, x_1 \p_1, \ds , x_{m} \p_{m} \rangle$ and $\langle (\mu_D - \chi)(\mf{t}) \rangle = \langle x_i \p_i - x_m \p_m - \chi_i \, | \, i \in [1,m-1] \rangle$, where we set $\p_i := \p / \p x_i$, $\U_{\chi}$ is generated by $\p_1 \cdots \p_{m} , x_1 \cdots x_{m}$ and $x_m \p_m$.

\subsection{The Dunkl embedding}

Let $\hr := \h \backslash \{ 0 \}$ and denote by $\algD (\hr)$ the ring of algebraic differential operators on $\hr$. In order to show that the spherical subalgebra of $H_\mbf{h}$ is isomorphic to a suitable quantum Hamiltonian reduction we realize $e H_{\mbf{h}} e$ as a subalgebra of $\algD (\hr)$ using the Dunkl embedding. Similarly, using the ``radial parts map'', we will also realize $\U_{\chi}$ as the same subalgebra of $\algD (\hr)$. The Dunkl embedding is the map $\Theta_{\mbf{h}} \, : \, H_{\mbf{h}} \longrightarrow \algD (\hr) \rtimes \Z_m$ defined by
$$
\Theta_{\mbf{h}}(y) = \frac{\txm{d}}{\txm{d} x} + \frac{m}{x} \sum_{i = 0}^{m-1} h_i e_i, \quad \Theta_{\mbf{h}}(x) = x \quad \textrm{ and } \quad \Theta_{\mbf{h}}(\varepsilon) = \varepsilon.
$$
The algebra $\algD (\hr) \rtimes \Z_m$ is filtered by order of differential operators, that is $\deg (\txm{d} / \txm{d} x) = 1$ and  $\deg (x) = \deg (\varepsilon) = 0$. If we define a filtration on $H_{\mbf{h}}$ by setting $\deg (y) = 1$ and $\deg (x) = \deg (\varepsilon) = 0$ then the map $\Theta_\mbf{h}$ is filter preserving. Localizing $H_{\mbf{h}}$ at the regular element $x$ provides an isomorphism
$$
\Theta_{\mbf{h}} \, : \, H_{\mbf{h}}[x^{-1}] \stackrel{\sim}{\longrightarrow} \algD (\hr) \rtimes \Z_m.
$$
Therefore $\Theta_{\mbf{h}}$ is injective. Applying the trivial idempotent produces
$$
\Theta_{\mbf{h}} \, : \, e H_{\mbf{h}} e \longrightarrow e \algD (\hr)e \simeq \algD (\hr)^{\Z_m}.
$$
Let us note that $\gr (H_{\mbf{h}}) \simeq \C [ x, y] \rtimes \Z_m$ and $\gr (e H_{\mbf{h}}e) \simeq \C [ x, y]^{\Z_m}$. Therefore $e H_{\mbf{h}} e$ is generated by $x^m e , xy e $ and $y^m e$. Since $\Theta_{\mbf{h}}(y e_i) = (\txm{d} / \txm{d} x + \frac{m}{x} h_{i})e_{i}$ we get $$
\Theta_{\mbf{h}}(y^m e) = \prod_{i = 1}^m \Bigl(\frac{\txm{d}}{\txm{d} x} + \frac{m}{x} h_i\Bigr) \quad \textrm{ and } \Theta_{\mbf{h}}(xy e) =  x \frac{\txm{d}}{\txm{d} x} + m h_m .
$$
We note that
\beq\label{eq:Dunkleoperator}
\Theta_{\mbf{h}}(y^m e)( x^r ) = \prod_{i = 1}^{m} (r - m + i + m h_i) x^{r - m}.
\eeq

\subsection{The radial parts map}

In this subsection we show that $\U_\chi \simeq \Theta_{\mbf{h}}(e H_{\mbf{h}} e)$. The isomorphism we describe is not new, it was first constructed by Holland \cite{Holland} (see also \cite{Kuwabara}), but we give it in order to fix parameters. There is a natural embedding $\h \hookrightarrow \Rp$ given by $x \mapsto (x, \ds, x)$. This defines a surjective morphism $\C [ \Rp] \rightarrow \C[\h], \, x_i \mapsto x$ which descends to a ``Chevalley isomorphism''
$$
\rho \, : \, \C [ \Rp]^{\T^{m-1}} \stackrel{\sim}{\longrightarrow} \C[\h]^{\Z_l}, \qquad x_1 \cdots x_{m} \mapsto x^m.
$$
Define a section
$$
\rho^{-1} \, : \,  \C[ \h] \longrightarrow \C[\Rp][x_i^{1/m} \, | \, i \in [1,m]] \quad \textrm{ by } \quad x^r \mapsto x_1^{r/m} \cdots x_{m}^{r/m}.
$$
This can be extended to a twisted Harish-Chandra morphism $\hat{\Rad}_\mbf{h} \, : \, \D (\Rp)^{\T} \longrightarrow \algD (\hr )^{\Z_m}$ given by
$$
\hat{\Rad}_\mbf{h} (D)(f) = \rho(\delta_\mbf{h}^{-1} D (\rho^{-1}(f) \delta_{\mbf{h}})), \quad \forall \, f \in \C[\h],
$$
where $\delta_{\mbf{h}} = \prod_{i = 1}^{m} x_i^{h_i + \frac{i-m}{m}}$. Calculating the action of $\hat{\Rad}_\mbf{h} (\p_{1} \cdots \p_{m})$ on $x^r$ and comparing with (\ref{eq:Dunkleoperator}) shows that
$$
\hat{\Rad}_\mbf{h} (m^m \cdot \p_{1} \cdots \p_{m}) = \Theta_{\mbf{h}}(y^m e).
$$
Similarly $\hat{\Rad}_\mbf{h} (x_1 \cdots x_{m}) = \Theta_{\mbf{h}}(x^m e)$ and $\hat{\Rad}_\mbf{h}( x_i \p_{i}) = \frac{1}{m} (x \frac{\txm{d}}{\txm{d} x} + m h_i + i - m)$. This implies that $\hat{\Rad}_{\mbf{h}}$ defines a surjection $\algD (\Rp)^{\T^{m-1}} \twoheadrightarrow \Theta_{\mbf{h}}(e H_{\mbf{h}} e)$. We fix
\beq\label{eq:chih}
\chi_i = h_i - h_m + \frac{i - m}{m}, \quad i \in [1,m-1].
\eeq
Then $\hat{\Rad}_\mbf{h}( x_i \p_{i} - x_m \p_m - \chi_i) = 0$ and $\hat{\Rad}_\mbf{h}$ descends to a surjective morphism
$$
\Rad_\mbf{h} \, : \, \U_\chi \longrightarrow \Theta_{\mbf{h}}(e H_{\mbf{h}} e).
$$
As above, $\algD (\Rp)$ is a filtered algebra by setting $\deg (\p_i) = 1$ and $\deg (x_i) = 0$ for $i \in [1,m]$. This induces a filtration on $\U_\chi$ and we see from the definitions that $\Rad_\mbf{h}$ is filter preserving. Therefore we get a morphism of associated graded algebras
$$
\gr \, \Rad_\mbf{h} \, : \, \gr (\U_\chi) \longrightarrow \gr (e H_{\mbf{h}} e).
$$
Now \cite[Proposition 2.4]{Holland} says that
$$
\gr (\U_\chi) = \C[ \mu^{-1}(0)]^{\T^{m-1}} \simeq \C[x,y]^{\Z_m} = \gr (e H_{\mbf{h}} e).
$$
This isomorphism is realized by $x_1 \cdots x_m \mapsto x^m$, $m^m \cdot y_1 \cdots y_m \mapsto y^m$ and $x_1 y_1 \mapsto \frac{1}{m} \cdot xy$. But we see from above that this is precisely what $\gr \, \Rad_\mbf{h}$ does to the principal symbols of the generators $m^m \cdot \p_{1} \cdots \p_{m}$, $x_1 \cdots x_m$ and $x_1 \p_1$ of $\U_\chi$. Therefore $\gr \, \Rad_\mbf{h}$ is an isomorphism and hence $\Theta_{\mbf{h}}^{-1} \circ \Rad_\mbf{h} \, : \, \U_\chi \stackrel{\sim}{\longrightarrow} e H_{\mbf{h}} e$ is a filtration preserving isomorphism.

\subsection{Localization of $H_{\mbf{h}}(\Z_m)$}

As noted in Proposition \ref{prop:Kleinianresolution}, the hypertoric varieties $Y(A,\delta)$ are all isomorphic provided $\delta$ does not belong to a wall in $\X_\Q$. Therefore, for any $\chi \in \mf{t}^*$, we may refer to the sheaf $\AW_\chi$ on the minimal resolution $\widetilde{\C^2 / \Z_m}$, but the reader should be aware that in doing so we have implicitly fixed an identification $\widetilde{\C^2 / \Z_m} = Y(A,\delta)$. Recall the union of hyperplanes $\mc{C} \subset \C^m$ defined in Lemma \ref{lem:sphericalequivalence}.

\begin{thm}\label{cor:Cherednikequiv}
Choose $\mbf{h} \in \C^m \backslash \mc{C}$ and let $\chi$ be defined by (\ref{eq:chih}). Write $\AW_{\mbf{h}} := \AW_{\chi}$ for the sheaf of $W$-algebras on $\widetilde{\C^2 / \Z_m}$. Then the functor $\Hom_{\Mod_F^{\mathrm{good}}(\AW_{\mbf{h}})}(\AW_{\mbf{h}} , - )$ defines an equivalence of categories
$$
\Mod^{\mathrm{good}}_F(\AW_{\mbf{h}}) \, \stackrel{\sim}{\longrightarrow} \, e H_{\mbf{h}} e \mmod
$$
with quasi-inverse $\AW_{\mbf{h}} \otimes_{e H_{\mbf{h}} e} - $. Moreover, the functor $H_{\mbf{h}} e \otimes_{e H_{\mbf{h}} e} \Hom_{\Mod_F^{\mathrm{good}}(\AW_{\mbf{h}})}(\AW_{\mbf{h}} , - )$ defines an equivalence of categories 
$$
\Mod^{\mathrm{good}}_F(\AW_{\mbf{h}}) \, \stackrel{\sim}{\longrightarrow} \, H_{\mbf{h}} \mmod
$$
with quasi-inverse $\AW_{\mbf{h}} \otimes_{e H_{\mbf{h}} e} e H_{\mbf{h}} \otimes_{H_{\mbf{h}}} - $.
\end{thm}

\begin{proof}
The condition $\chi_i \neq \chi_j$ for $i \neq j \in [1,m-1]$ translates, via (\ref{eq:chih}), into $h_i - h_j + \frac{i - j}{m} \neq 0$ for all $i \neq j \in [1,m-1]$. Similarly, the condition $\chi_i \neq 0$ for all $i \in [1,m-1]$ translates into $h_i - h_m + \frac{i - m}{m} \neq 0$ for all $i \in [1,m-1]$. Therefore the linear map $\C^m \rightarrow \C^{m-1}$ defined by (\ref{eq:chih}) maps the union of hyperplanes $\mc{C}$ onto 
$$
\{ \chi \in \C^{m-1} \ | \ \chi_i = \chi_j  \, (i \neq j \in [1,m-1]) \, \textrm{ or } \, \chi_i = 0, \, (i \in [1,m-1]) \},
$$
which is precisely the union of the G.I.T. walls in $\C^{m-1}$. Therefore Lemma \ref{lem:sphericalequivalence} implies that $\U_\chi$ has finite global dimension when $\chi$ lies in the interior of some G.I.T. cone $C$. Now the theorem follows from Corollary \ref{cor:finiteglobaldim}.
\end{proof}

\begin{remark}
In the above situation it is possible to explicitly calculate the sets $\mc{Q}_{\chi}$ and hence describe the partial ordering on comparability classes as defined in (\ref{sec:QHRone}). However the answer is not very illuminating.
\end{remark}

Finally, we would just like to note the various forms in which the rational Cherednik algebra $H_{\mbf{h}}(\Z_m)$ appears in the literature. It is isomorphic to the deformed preprojective algebra of type $A$ as studied by Crawley-Boevey and Holland in \cite{CrawleyBoeveyHolland}. 
It is well-known that its spherical subalgebra $e H_{\mbf{h}}(\Z_m) e$ coincides with a ``generalized $U(\mf{sl}_2)$-algebra'', as studied by Hodges \cite{Hodges} and Smith \cite{Smith}. Combining this fact with Premet's results in \cite{Premet} shows that $e H_{\mbf{h}}(\Z_m) e$ is also isomorphic to the finite $\mathcal{W}$-algebra associated to $\mf{gl}_m(\C)$ at a subregular nilpotent element. 
Recently, in \cite{Losev} Losev has constructed explicit isomorphisms between the spherical subalgebra of certain rational Cherednik algebras and their related finite $\mathcal{W}$-algebras, which as a special case, gives the above mentioned isomorphism. 

Musson \cite{Musson} and Boyarchenko \cite{Boyarchenko} studied a certain localization of $e H_{\mbf{h}}(\Z_m) e$ by using the formalism of directed algebras (or $\Z$-algebras). Analogous localizations for finite $\mathcal{W}$-algebras were established by Ginzburg in \cite{HC-bimod}. Recently, Dodd and Kremnizer \cite{DK} described a localization theorem for finite $\mathcal{W}$-algebras in the spirit of Kashiwara-Rouquier, and in particular for the finite $W$-algebra isomorphic to $e H_{\mbf{h}}(\Z_m) e$. However, their result is via a different quantum Hamiltonian reduction than the one used in Theorem \ref{cor:Cherednikequiv}.


In \cite{Kuwabarapreprint}, the second author gives an explicit description of the standard modules $\Delta(e_i)$ and simple modules $L(e_i)$ as sheaves of $\AW_{\mbf{h}}$-modules on the minimal resolution.

\section*{Acknowledgments}

The authors would like to express their gratitude to Professors Iain Gordon, Masaki Kashiwara and Raphael Rouquier for their helpful comments and many fruitful discussions. They would also like to thank Professors Brend Sturmfels and Tam\'as Hausel for outlining the proof of Lemma \ref{prop:compint}. The first author thanks Maurizio Martino and Olaf Schn\"urer for stimulating discussions. The second author was partially supported by Grant-in-Aid for Young Scientists (B) 21740013, and by GCOE `Fostering top leaders in mathematics', Kyoto University. He was also partially supported by Basic Science Research Program through the National Research Foundation of Korea (NRF) grant funded by the Korea government (MEST)(2010-0019516).

\def\cprime{$'$}


\end{document}